\documentclass[12pt]{article}

\usepackage[margin=1.3in, top=1.3in, bottom=1.3in]{geometry}

\usepackage[utf8]{inputenc} 
\usepackage[T1]{fontenc}    
\usepackage{url}            
\usepackage{booktabs}       
\usepackage{amsfonts}       
\usepackage{nicefrac}       
\usepackage{microtype}      

\usepackage{amsmath,amssymb,amsthm,xcolor}
\usepackage{array}
\usepackage{float}

\theoremstyle{definition}
\newtheorem{example}{Example}
\newtheorem{remark}{Remark}
\newtheorem{definition}{Definition}
\newtheorem{assumption}{Assumption}
\theoremstyle{plain}
\newtheorem{theorem}{Theorem}
\newtheorem{lemma}{Lemma}
\newtheorem{corollary}{Corollary}

\newtheorem{proposition}{Proposition}

\usepackage{graphicx}
\usepackage{subcaption}
\usepackage{mdframed}
\usepackage{lipsum}
\usepackage{wrapfig}
\usepackage[hidelinks,hypertexnames=false]{hyperref}
\usepackage{enumitem}
\allowdisplaybreaks
\usepackage{epstopdf}
\usepackage{algorithmic}
\usepackage{booktabs}
\usepackage{changepage}
\usepackage{amssymb}

\usepackage{cleveref}
\Crefname{assumption}{assumption}{assumptions}

\DeclareMathOperator{\diag}{diag}
\DeclareMathOperator{\rank}{rank}

\crefalias{enumi}{proposition}

\crefname{equation}{}{}
\hypersetup{colorlinks=true,citecolor=blue, linkcolor=blue}

\title{\LARGE Certifying the absence of spurious local minima at infinity}

\begin{document}

\author{\large C\'edric Josz\thanks{\url{cj2638@columbia.edu}, IEOR, Columbia University, New York. Research supported by NSF EPCN grant 2023032 and ONR grant N00014-21-1-2282.} \and Xiaopeng Li\thanks{\url{xl3040@columbia.edu}, IEOR, Columbia University, New York.}}
\date{}

\maketitle

\begin{center}
    \textbf{Abstract}
    \end{center}
    \vspace*{-3mm}
 \begin{adjustwidth}{0.2in}{0.2in}
~~~~When searching for global optima of nonconvex unconstrained optimization problems, it is desirable that every local minimum be a global minimum. This property of having no spurious local minima is true in various problems of interest nowadays, including principal component analysis, matrix sensing, and linear neural networks. However, since these problems are non-coercive, they may yet have spurious local minima at infinity. The classical tools used to analyze the optimization landscape, namely the gradient and the Hessian, are incapable of detecting spurious local minima at infinity. In this paper, we identify conditions that certify the absence of spurious local minima at infinity, one of which is having bounded subgradient trajectories. We check that they hold in several applications of interest.
\end{adjustwidth} 
\vspace*{3mm}

\section{Introduction}\label{sec:introduction}

The idea that the absence of spurious local minima alone does not guarantee the success of first-order methods was first expressed in the context of binary classification in the mid-nineties. It was shown that gradient trajectories are bounded if the objective function satisfies several technical conditions tailored to the problem at hand \cite[Theorems 3.6-3.8]{blackmore1996local}. This property was referred to as having no attractors at infinity. More recently, it was proved that adding an exponential neuron to a wide class of neural networks eliminates all spurious local minima \cite{liang2018adding}, but it was soon realized that this procedure simply sends them to infinity \cite{sohl2019eliminating}. These results suggest that besides spurious local minima, a certain notion of spurious local minima at infinity also affects the convergence of first-order methods to global optima. However, the current optimization literature lacks a precise definition of local minima at infinity, and, accordingly, there is little theoretical understanding of them. Worse still, classical tools for landscape analysis, such as the gradient and the Hessian, cannot detect spurious local minima at infinity even in simple scenarios (see \Cref{eg:infty-spurious}), let alone handle nonsmooth functions without a gradient.

\begin{example}
\label{eg:infty-spurious}
Consider an instance of matrix completion problem, i.e., minimize
\begin{equation*}
    f(x_1,x_2,y_1,y_2) := (x_1y_1-1)^2+(x_2y_1-1)^2+(x_2y_2-1)^2. 
\end{equation*}
By solving $\nabla f(x_1,x_2,y_1,y_2)=0$, the set of critical points of $f$ can be decomposed into four connected components: 
\begin{gather*}
    C_1 = \{(x_1,x_2,y_1,y_2) = (t,t,1/t,1/t)\,\vert\,t\in\mathbb{R}\setminus \{0\} \}, \\
    C_2 = \{(x_1,x_2,y_1,y_2) = (t,0,1/t,-1/t)\,\vert\,t\in\mathbb{R}\setminus \{0\}\}, \\
    C_3 = \{(x_1,x_2,y_1,y_2) = (t,-t,0,-1/t)\,\vert\,t\in\mathbb{R}\setminus \{0\}\}, \\
    C_4 = \{(x_1,x_2,y_1,y_2) = (0,0,0,0)\}. 
\end{gather*}
The critical values are $f(C_1)=\{0\}$, $f(C_2)=f(C_3)=\{2\}$, and $f(C_4)=\{3\}$. Furthermore, $C_1$ is the set of global minima, and by computing the Hessian $\nabla^2 f$, we find that it has positive and negative eigenvalues at all points in $C_2$, $C_3$ and $C_4$. Therefore, $f$ has no spurious local minima and all saddle points are strict \cite[Definition 2]{lee2016}. One would expect first-order methods like gradient descent to converge to a global minimum for almost all initial points \cite[Theorem 11]{lee2016}. However, the numerical experiments in \Cref{fig:spurious_inf} show otherwise. This is because the function is not coercive.
\begin{figure}[htbp]
    \centering
    \includegraphics[width=.6\textwidth]{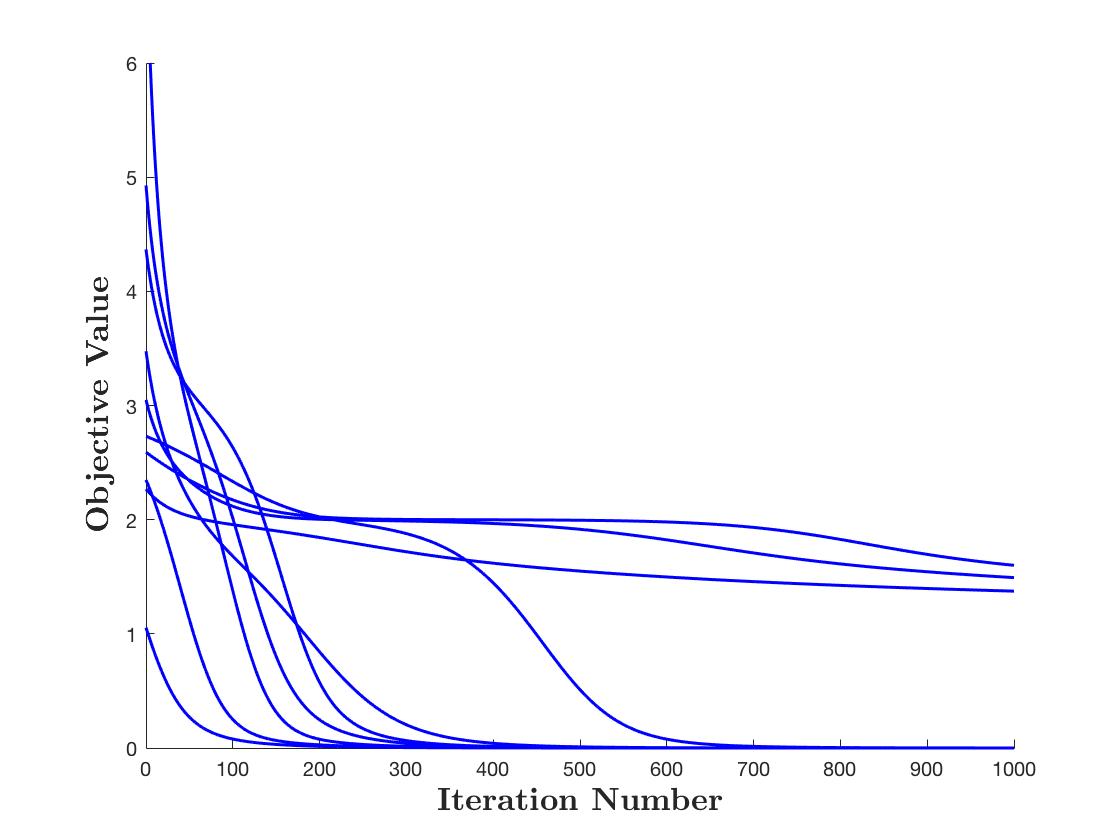}
    \caption{Gradient method initialized uniformly at random in $[-1,1]^4$ with constant step size $0.01$ sometimes gets stuck at a spurious local minimum at infinity (3 among 10 trails in the experiment).}
    \label{fig:spurious_inf}
\end{figure}
\end{example}

Two newly proposed concepts related to spurious local minima at infinity are setwise local minima \cite{joszneurips2018} and spurious valleys \cite{venturi2019spurious}. Setwise local minima \cite[Definition 2.5]{joszneurips2018} generalize the notion of local minima from points to compact sets. The first author and co-authors recently established that the uniform limit (on all compact subsets) of a sequence of continuous functions which are devoid of spurious setwise local minima is itself devoid of spurious strict setwise local minima \cite[Proposition 2.7]{joszneurips2018}. However, due to the boundedness assumption, setwise local minima cannot be directly used to study spurious local minima at infinity. Spurious valleys \cite[Definition 1]{venturi2019spurious} do have the potential to handle spurious local minima at infinity but they fail to detect them when there are flat regions, such as in the ReLU network with one-hidden layer $(x_1,x_2) \mapsto (x_2\max\{x_1,0\}-1)^2$ (see \Cref{fig:relu}). Spurious valleys also rely on the notion of path-connectedness, which is actually not necessary for defining spurious local minimum at infinity. In this paper, we extend the concept of setwise local minima by relaxing the boundedness assumption. This enables us to define spurious local minima at infinity as unbounded setwise local minima over which the infimum of the objective function is greater than the global infimum. It also allows us to handle classical spurious local minima and flat regions in a unified way. 

\begin{figure}[htbp]
    \centering
    \includegraphics[width=.6\textwidth]{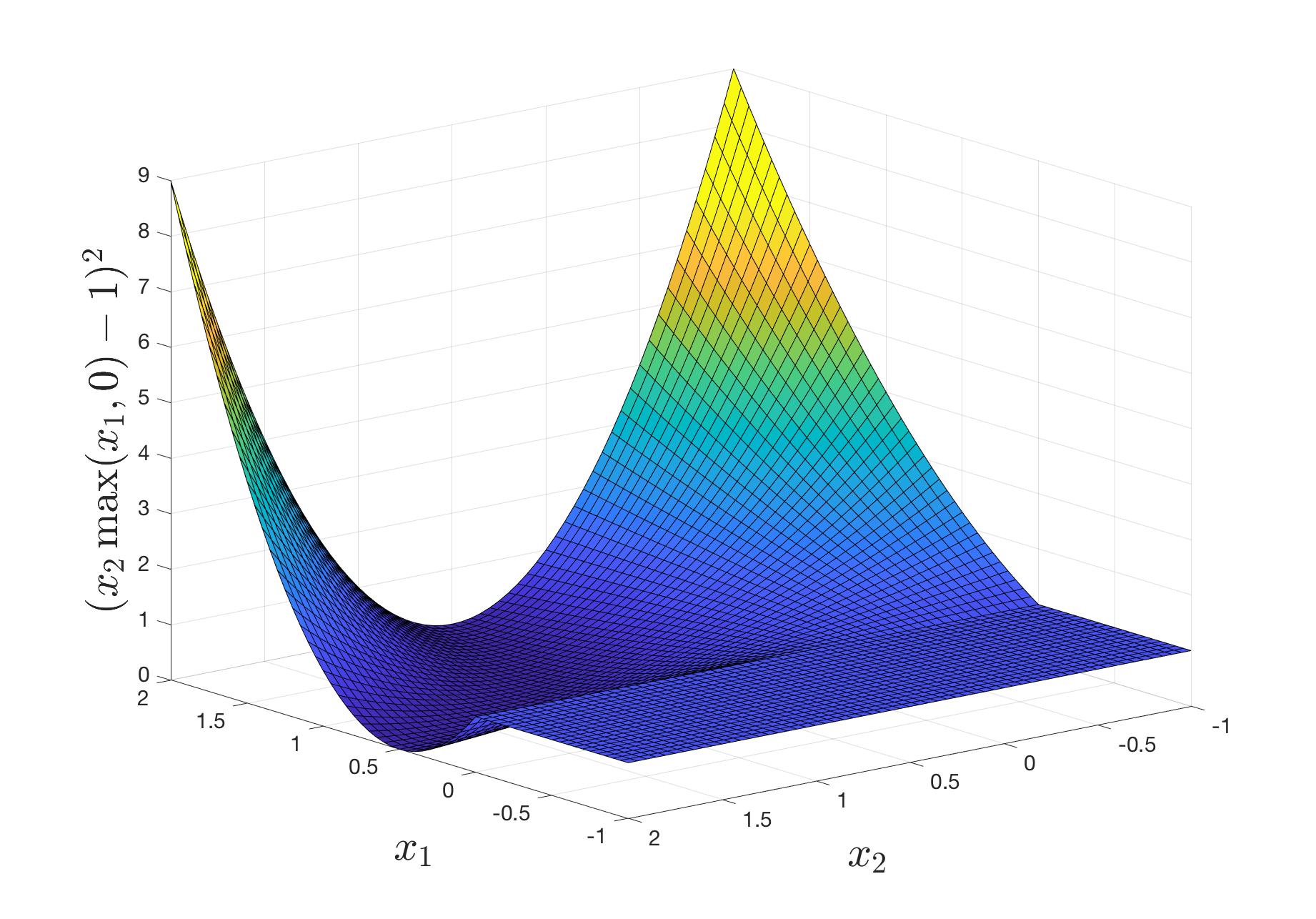}
    \caption{Function devoid of spurious valleys containing a spurious local minimum at infinity.}
    \label{fig:relu}
\end{figure}

An existing strategy to analyze the landscape of non-coercive functions is to construct a strictly decreasing path to a global minimum from any initial point. Such a path was shown to exist in half-rectified neural network \cite{freeman2017topology}. This strategy is used to prove the existence of spurious local minima in neural networks with almost all nonlinear activations \cite{ding2022suboptimal}. It also explains the phase transition from the existence of sub-optimal basins in narrow networks to their disappearance in wide networks \cite{li2022benefit}. Finally, it is used to prove the absence of spurious valleys for over-parametrized one-hidden layer neural network \cite{venturi2019spurious}. However, such a strategy needs to be tailored to each application since one needs to select a particular path for each specific loss function. In this paper, we instead develop a theory allowing one to use a common decreasing path - subgradient trajectory - to analyze the landscape in various different contexts. We can then rule out spurious setwise local minima (and thus those at infinity) for a general class of functions. Our main result is as follows.

\begin{theorem}\label{thm:main-intro}
Suppose a locally Lipschitz function is bounded below, admits a chain rule, has finitely many critical values, and has bounded subgradient trajectories. Then it has no spurious local minima if and only if it has no spurious setwise local minima. 
\end{theorem}

The above statement is meant to help readers get a first taste of our main result in this paper. Precise definitions and a detailed mathematical background will be given in the main body\footnote{The terminology in the theorem will be given in \Cref{def:locally-lipschitz} (locally Lipschitz), \Cref{def:chain-rule} (chain rule), \Cref{def:bdd-subgrad} (bounded subgradient trajectories), and \Cref{def:setwise-local-min} (setwise local minimum).}, along with a discussion on the role of its assumptions. Let us mention already that two of its assumptions, namely those regarding the chain rule and critical values, automatically hold for functions definable in an o-minimal expansion of the real field \cite{van1998tame} (by \cite[Proposition 2 (iv)]{bolte2020conservative} and the definable Morse-Sard theorem \cite[Corollary 9 (ii)]{bolte2007clarke}). This includes semi-algebraic, globally subanalytic, and log-exp functions, and importantly, many applications of interest nowadays \cite[Section 4.1]{bolte2020conservative}. The locally Lipschitz and lower bounded assumptions usually come for free in applications, so that in practice the sole assumption that one needs to check for is that subgradient trajectories are bounded. \Cref{thm:main-intro} thus serves as a handy device to conclude that there are no spurious setwise local minima for a family of functions that are widely used in machine learning, especially in deep neural networks and matrix sensing. We summarize the problems that we are going to consider in the following corollary. 

\begin{corollary}\label{cor:main}
The following problems have no spurious local minima at infinity: 
\begin{enumerate}
    \item deep linear neural network 
    $$\inf_{W_1,\ldots,W_L} ~~~ \lVert W_L\cdots W_1X - Y\rVert_F^2;$$
    \item one dimensional deep neural network with sigmoid activation function $\sigma$
    $$\inf_{w_1,\ldots,w_L} ~~~ (w_L\sigma(w_{L-1}\cdots \sigma(w_1x))-y)^2;$$
    \item matrix recovery with restricted isometry property (RIP)
    $$\inf_{X,Y} ~~~ \frac{1}{2m}\sum_{i=1}^m(\langle A_i,XY^{\rm T}\rangle_F-b_i)^2;$$
    \item nonsmooth matrix factorization where $\rank(M) = 1$ and $M_{ij} \neq 0$
    $$ \inf_{x,y} ~~~ \sum_{i=1}^m \sum_{j=1}^n \lvert x_iy_j-M_{ij} \rvert.$$
\end{enumerate}
\end{corollary}

Again, the statement above aims at giving readers some feeling on what type of functions we are considering. More rigorous descriptions of the applications will be given in the main body.

The paper is organized as follows. \Cref{sec:background} contains background material on setwise local minima, the Clarke subdifferential, and subgradient trajectories. \Cref{sec:proofmain} contains the proof of our main result, namely \Cref{thm:main-intro}. Finally, \Cref{sec:applications} contains applications of our main result as delineated in \Cref{cor:main}. 

\section{Background}\label{sec:background}

This section contains prerequisites for the proof of \Cref{thm:main-intro} in the next section. Throughout this paper, unless otherwise specified, we always equip the Euclidean space $\mathbb{R}^n$ with an inner product $\langle \cdot,\cdot\rangle$ and its induced norm $\|\cdot\|:=\sqrt{\langle \cdot,\cdot\rangle}$. 

\subsection{Setwise local minimum}\label{subsec:setwise}

In this subsection, we present the formal definitions and some useful properties of setwise local minimum and local minimum at infinity mentioned in \Cref{sec:introduction}. We first review the classical definition of local and global minima. Throughout this paper, $B(x,\epsilon)$ denotes the open ball centered at $x\in\mathbb{R}^n$ with radius $\epsilon>0$. 

\begin{definition}
\label{def:global-local-min}
A point $x\in\mathbb{R}^n$ is a \textit{local minimum} (respectively, \textit{global minimum}) of a function $f:\mathbb{R}^n\rightarrow\mathbb{R}$ if $f(x)\leqslant f(y)$ for all $y\in B(x,\epsilon)$ for some $\epsilon>0$ (respectively, $y\in\mathbb{R}^n$). A local minimum is \textit{spurious} if it is not a global minimum.
\end{definition}

From \Cref{def:global-local-min}, one can see the definition of a local minimum only considers the landscape of a function at any finite point. To discuss the function landscape at infinity, we generalize the notion of setwise local minimum first proposed in \cite[Definition 2.5]{joszneurips2018}. 

\begin{definition}[Setwise local minimum]\label{def:setwise-local-min}
A nonempty closed subset $S\subset\mathbb{R}^n$ is a \textit{setwise local minimum} of a continuous function $f:\mathbb{R}^n\rightarrow\mathbb{R}$ if there exists an open set $U\subset\mathbb{R}^n$ such that $S\subset U$ and $f(x)\leqslant f(y)$ for all $x\in S$, $y\in U\setminus S$. \end{definition}

It is easy to see that a local minimum is a setwise local minimum by taking $S$ to be a singleton. We also define a \textit{strict setwise local minimum} by replacing $f(x)\leqslant f(y)$ with $f(x)<f(y)$ in \Cref{def:setwise-local-min}. 

\begin{definition}[Setwise global minimum]\label{def:setwise-global-min}
A subset $S$ of $\mathbb{R}^n$ is a \textit{setwise global minimum} of a function $f:\mathbb{R}^n\rightarrow\mathbb{R}$ if $S$ is a setwise local minimum of $f$ and $\inf_{S} f = \inf_{\mathbb{R}^n} f$. 
\end{definition}

Note that $\inf_S f$ is a shorthand for $\inf_{x\in S} f(x)$, and similar for $\sup_S f$ and $\max_S f$. Setwise local minima include setwise global minima as a special case, and we say a setwise local minimum is \textit{spurious} if it is not a setwise global minimum. Note that \Cref{def:setwise-local-min} is not exactly the same as \cite[Definition 2.5]{joszneurips2018} because we do not require a setwise local minimum to be a compact set. In other words, a setwise local minimum can be either bounded or unbounded, and we say a (spurious) setwise local minimum is a (spurious) \textit{local minimum at infinity} if it is unbounded. For example, consider the loss function of a one-hidden layer neural network with sigmoid activation $\sigma$ and two data points $(1,1)$ and $(-1,-3)$ in \Cref{fig:setwise}. One can see that $S$  is a setwise local minimum (in particular, a local minimum at infinity) and $U$ is the corresponding open set in \Cref{def:setwise-local-min}. Finally, observe that $\mathbb{R}^n$ is a strict setwise local minimum at infinity of any function.

\begin{figure}[htbp]
	\centering
	\includegraphics[width=.6\textwidth]{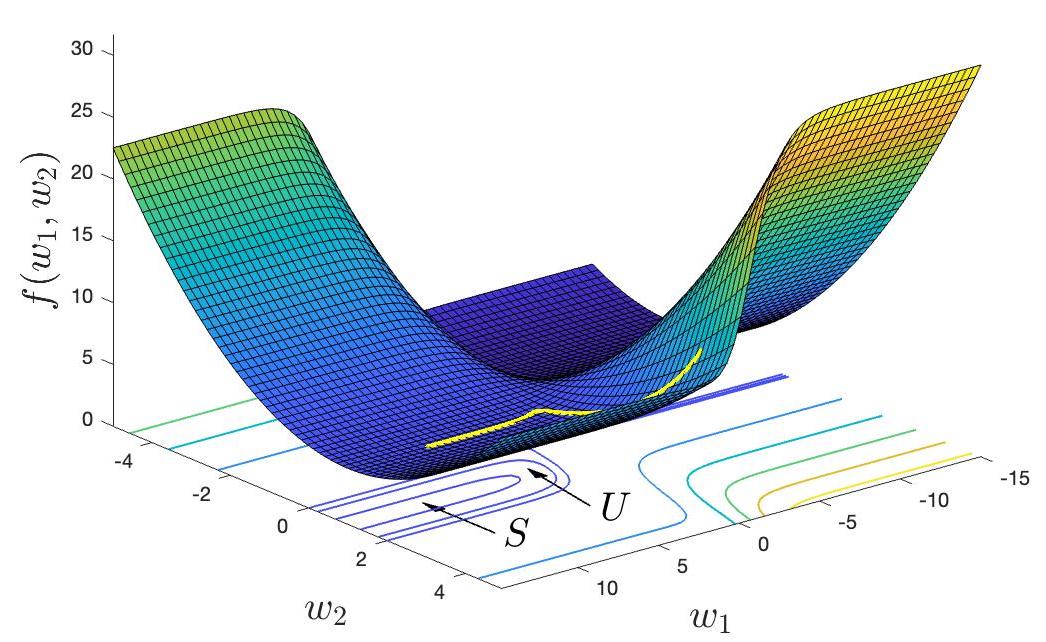}
	\caption{Local minimum at infinity of  $f(w_1,w_2)=\frac{1}{2}[(w_2\sigma(w_1)-1)^2+(w_2\sigma(-w_1)+3)^2]$.}
	\label{fig:setwise}
\end{figure}

Now we introduce one of the most useful properties of setwise local minima in \Cref{lem:bdd-const}. This property is intuitive and will be used in different scenarios throughout this paper. Let $\bar{S}$, $S^\circ$, and $\partial S := \bar{S} \setminus S^\circ$ respectively denote the closure, interior, and boundary of a subset $S$ of $\mathbb{R}^n$.

\begin{lemma}\label{lem:bdd-const}
\label{lemma:boundary}
If $S$ is a setwise local minimum of a continuous function $f:\mathbb{R}^n\rightarrow\mathbb{R}$, then $f(x)=\sup_S f$ for all $x\in\partial S$.
\end{lemma}
\begin{proof}
See \Cref{app:lem:bdd-const}. 
\end{proof}

It is worth relating our notion of setwise local minimum to the concept of \textit{valley} proposed in \cite[Definition 1]{venturi2019spurious}. A \textit{valley} of a function $f:\mathbb{R}^n\rightarrow\mathbb{R}$ is defined as a path-connected component\footnote{A subset $S$ of $\mathbb{R}^n$ is path-connected if for all $x,y\in S$, there exists a continuous function $\gamma:[0,1]\rightarrow S$ such that 
$\gamma(0)=x$ and $\gamma(1)=y$. A maximal path-connected connected set is called a path-connected component. Path-connected components can be viewed as equivalence classes over a set.} of a sublevel set of $f$. These two definitions are distinct in general. The interval $[-1,1]$ is a setwise local minimum of  $f:\mathbb{R}\rightarrow \mathbb{R}$ defined by $f(x) := 0$ for all $x \in \mathbb{R}$ yet it is not a valley. Conversely, $X:=\{(x_1,x_2) \in (\mathbb{R}\setminus \{0\}) \times \mathbb{R}~\vert~ x_2 = \sin(1/x_1)\}$ is a valley of $f:\mathbb{R}^2\rightarrow \mathbb{R}$ where $f$ is defined as the distance between $x$ and $X$, yet it is not a setwise local minimum since it is not closed. (The sublevel set of $f$ corresponding to the value zero is composed of two path-connected components, namely $X$ and $\{0\}\times[-1,1]$, whose union is $\overline{X}$.) Under some mild conditions, their relation can be summarized in \cref{prop:valley}. 

\begin{proposition}\label{prop:valley}
For a continuous function from $\mathbb{R}^n$ to $\mathbb{R}$,
\begin{enumerate}[label=(\alph*),ref=\theproposition~(\alph*)]
    \item a path-connected component of a strict setwise local minimum (respectively, setwise local minimum) is a valley (respectively, subset of a valley); \label{prop:valley:min-valley}
    \item a connected component\footnote{A subset $S$ of $\mathbb{R}^n$ is disconnected if there exist nonempty disjoint open (in $S$) sets $A$ and $B$ such that $S = A\cup B$. It is connected if it is not disconnected. A maximal connected set is called a connected component.} of a sublevel set which has finitely many connected components is a strict setwise local minimum. \label{prop:valley:valley-min}
\end{enumerate}
\end{proposition}
\begin{proof}
See \Cref{app:prop:valley}. 
\end{proof}

\begin{remark}\label{rem:oscill}
The assumption on finiteness of connected components is necessary, or else counterexample may occur when the function is oscillatory. For example, 
\begin{equation*}
    f(x) := \begin{cases}
        0 & \mathrm{if}~x\leqslant 0, \\
        x^2\sin\frac{1}{x} & \mathrm{if}~x>0.
    \end{cases}
\end{equation*}
The function $f$ is continuous on $\mathbb{R}$, but the sublevel set $\{x\in\mathbb{R}\,\vert\,f(x)\leqslant 0\}$ has infinitely many connected components. Take a connected component $C_1=(-\infty,0]$ (also path-connected, thus a valley), and it is not a setwise local minimum because for any open set $U$ containing $C_1$, there exists some $x_0\in N$ such that $f(x_0)<0$. 
\end{remark}

Finally, we discuss the case of coercive functions. Recall that a function $f:\mathbb{R}^n\rightarrow \mathbb{R}$ is coercive if $f(x)\to\infty$ as $\|x\|\to\infty$. 

\begin{proposition}\label{prop:char-coercive}
If a continuous function from $\mathbb{R}^n$ to $\mathbb{R}$ is coercive, then it has no spurious local minima at infinity.
\end{proposition}
\begin{proof}
See \Cref{app:prop:char-coercive}. 
\end{proof}

In many statistical learning problems, the loss functions without regularizer are usually not coercive, so spurious local minima at infinity may exist. Therefore, it is important to develop some device to check whether spurious local minima exist or not so that optimization algorithms can be designed to avoid getting trapped in them.

\subsection{Clarke subdifferential}\label{subsec:clarke}

In this subsection, we will review some concepts and results on generalized derivative in the sense of Clarke \cite[p. 27]{clarke1990}, since \Cref{thm:main-intro} also considers nonsmooth functions.

\begin{definition}\label{def:locally-lipschitz}
A function $f:\mathbb{R}^n\rightarrow\mathbb{R}^m$ is \textit{locally Lipschitz} if for all $a \in \mathbb{R}^n$, there exist positive constants $r$ and $L$ such that 
	\begin{equation*}
		\forall\, x,y \in B(a,r), ~~~ \|f(x)-f(y)\| \leqslant L \|x-y\|.
	\end{equation*}
\end{definition}

Notice that for a locally Lipschitz function, by \cite[Theorem 3.2]{evans2018measure}, the derivative exists almost everywhere. Without any assumption on convexity, in order to ensure the existence of a subdifferential, we adopt the notion of Clarke subdifferential. 

\begin{definition}
\label{def:Clarke}
\cite[p. 27]{clarke1990} Let $f:\mathbb{R}^n \rightarrow \mathbb{R}$ be a locally Lipschitz function. The \textit{Clarke subdifferential} is the set-valued mapping $\partial f$ from $\mathbb{R}^n$ to the subsets of $\mathbb{R}^n$ defined for all $x \in \mathbb{R}^n$ by
    \begin{equation*}
    \partial f(x) := \{ s \in \mathbb{R}^n ~\vert~ f^\circ(x,d) \geqslant \langle s , d \rangle, ~ \forall\, d\in \mathbb{R}^n \} 
\end{equation*}
where
\begin{equation*}
    f^\circ(x,d) := \limsup_{\tiny\begin{array}{c} y\rightarrow x \\
    t \searrow 0
    \end{array}
    } \frac{f(y+td)-f(y)}{t}.
\end{equation*}
\end{definition}

It is well known that for any locally Lipschitz function $f$ and any $x\in\mathbb{R}^n$, the Clarke subdifferential $\partial f(x)$ is a nonempty, convex, and compact set \cite[Proposition 2.1.2(a)]{clarke1990}. Similar to differentiable functions, a point $x\in\mathbb{R}^n$ is called (Clarke) \textit{critical point} if $0\in\partial f(x)$. A real number $y$ is called a (Clarke) \textit{critical value} of $f$ if there exists a (Clarke) critical point $x\in\mathbb{R}^n$ of $f$ such that $f(x)=y$.

\subsection{Subgradient trajectory}

In this subsection, we will introduce some basic concepts and fundamental properties related to subgradient trajectories. 
\begin{definition}
    \label{def:absolute_continuity} 
    \cite[Definition 1 p. 12]{aubin1984differential}
    Given two real numbers $a<b$, a function $x:[a,b]\rightarrow\mathbb{R}^n$ is \textit{absolutely continuous} if for all $\epsilon>0$, there exists $\delta>0$ such that, for any finite collection of disjoint subintervals $[a_1,b_1],\hdots,[a_m,b_m]$ of $[a,b]$ such that $\sum_{i=1}^m b_i-a_i \leqslant \delta$, we have $\sum_{i=1}^m \|x(b_i) - x(a_i)\| \leqslant \epsilon$.
\end{definition}

By virtue of \cite[Theorem 20.8]{nielsen1997introduction}, $x:[a,b]\rightarrow\mathbb{R}^n$ is absolutely continuous if and only if it is differentiable almost everywhere on $(a,b)$, its derivative $x'$ is Lebesgue integrable, and $x(t) - x(a) = \int_a^t x'(\tau)d\tau$ for all $t\in [a,b]$. Given a non-compact interval $I$ of $\mathbb{R}$, $x:I\rightarrow \mathbb{R}^n$ is absolutely continuous if it is absolutely continuous on all compact subintervals of $I$.

An absolutely continuous function $x:[0,\infty)\rightarrow\mathbb{R}^n$ is called a \textit{subgradient trajectory} of $f:\mathbb{R}^n\rightarrow\mathbb{R}$ starting at $x_0\in\mathbb{R}^n$ if it satisfies the following differential inclusion with initial condition: 
\begin{equation}\label{eq:ivp-subgrad}
    x'(t) \in -\partial f(x(t)),\quad \text{for almost every $t\geqslant 0$},\quad x(0)=x_0,
\end{equation}
where ``almost every'' means all elements except for those in a set of zero measure. 

However, a subgradient trajectory may not always exist for arbitrary $f$, even if $f$ is a smooth function. Let $f(x)=-\frac{1}{3}x^3$ and $x_0=1$, then it is easy to see $x(t)=\frac{1}{1-t}$ is the unique solution for $t\in[0,1)$ and it cannot be extended to an absolutely continuous function on $[0,\infty)$ due to the singularity at $t=1$. 
In this case, one would seek a family of functions including many loss functions arising in applications that guarantee the existence of a subgradient trajectory. We say a function $f:\mathbb{R}^n\rightarrow\mathbb{R}$ is \textit{bounded below} if $\inf_{\mathbb{R}^n}f=c>-\infty$. It was shown in \cite[Theorem 3.2]{marcellin2006evolution} that a primal lower nice function bounded below by a linear function suffices. However, in general it is not easy to check whether those nonconvex functions in statistical learning problems are primal lower nice. For easily checkable conditions, the following result generalized from \cite[Proposition 2.3]{santambrogio2017euclidean} for differentiable functions tells us that a locally Lipschitz function bounded below also suffices. 

\begin{proposition}\label{prop:exist}
If $f:\mathbb{R}^n\rightarrow\mathbb{R}$ is locally Lipschitz and bounded below, then there exists a subgradient trajectory of $f$ starting at arbitrary $x_0\in\mathbb{R}^n$. 
\end{proposition}
\begin{proof}
See \Cref{app:prop:exist}. 
\end{proof}

We remark here that with \Cref{prop:exist}, one can recover Ekeland's variational principle \cite[Corollary 2.3]{ekeland1974variational} \cite[Corollary]{hiriart1983short} for locally Lipschitz lower bounded functions with a chain rule (see \cite[Theorem 3.1]{ha2003ekeland} for an extension to lower semi-continuous lower bounded functions). Indeed, \Cref{prop:exist} implies that for all $\epsilon>0$, there exists $(x,s)  \in \mathrm{graph} ~ \partial f$ such that $f(x) \leqslant \inf f + \epsilon$ and $\|s\| \leqslant \epsilon$.\footnote{This follows from the formula $f(x(t)) - \inf f \geqslant \int_t^\infty d(0,\partial f(x(\tau)))^2d\tau$ where $d(x,X) := \inf_{y\in X} \|x-y\|$ (see \cite[Lemma 5.2]{davis2020stochastic} and \cite[Proposition 4.10]{drusvyatskiy2015curves}).} Note that \Cref{prop:exist} only guarantees the existence of a solution to \Cref{eq:ivp-subgrad} for all $t\geqslant 0$, but the solution $x(t)$ could go to infinity as $t\to\infty$. This motivates the following definition.

\begin{definition}\label{def:bdd-subgrad}
A locally Lipschitz lower bounded function $f:\mathbb{R}^n\rightarrow\mathbb{R}$ has \textit{bounded subgradient trajectories} if for any $x_0\in\mathbb{R}^n$, there exists a subgradient trajectory $x$ of $f$ starting at $x_0$ and a constant $r>0$, such that $\|x(t)\|\leqslant r$ for all $t\geqslant 0$. 
\end{definition}

Finally, notice that when $f$ is continuously differentiable, by \cite[Proposition 2.2.4]{clarke1990}, \Cref{eq:ivp-subgrad} reduces to the classical Cauchy problem of differential equation
\begin{equation*}
    x'(t) = -\nabla f(x(t)), \quad \text{for all $t \geqslant 0$,} \quad x(0)=x_0. 
\end{equation*}
and subgradient trajectory reduces to gradient trajectory by imposing $x$ to be continuously differentiable. Recall the descent property of gradient trajectories \cite[Proposition 17.1.1]{attouch2014variational}, i.e., $f \circ x$ is a decreasing function for any gradient trajectory $x$ of $f$. We want this nice property to hold even in the general case when $f$ is only locally Lipschitz. We adopt the notion of chain rule in \cite[Definition 5.1]{davis2020stochastic}. Note that functions admitting a chain rule are also referred to as path differentiable \cite[Definition 3]{bolte2020conservative}.

\begin{definition}\label{def:chain-rule}
Let $f:\mathbb{R}^n\rightarrow\mathbb{R}$ be locally Lipschitz. We say $f$ \textit{admits a chain rule} if for any absolutely continuous function $x:[0,\infty)\rightarrow \mathbb{R}^n$, we have
\begin{equation*}
	(f\circ x)'(t) = \langle v, x'(t)\rangle, \quad \forall\, v\in\partial f(x(t)),
\end{equation*}
for almost every $t\in [0,\infty)$. 
\end{definition}

Thus, for any locally Lipschitz function that admits a chain rule, by \cite[Lemma 5.2]{davis2020stochastic}, the function value is always decreasing in time along the subgradient trajectory. A detailed discussion on what class of functions admits a chain rule can be found in \cite{bolte2020conservative}. Note that general Lipschitz functions are far from admitting a chain rule since they generically have a maximal Clarke subdifferential \cite{daniilidis2020pathological,borwein2000lipschitz,daniilidis2019linear}.

\section{Proof of \texorpdfstring{\Cref{thm:main-intro}}{}}\label{sec:proofmain}

This section contains the proof of the main result, i.e., \Cref{thm:main-intro}. After the proof, we will explain the necessity of the assumptions in \Cref{thm:main-intro} by raising some counterexamples. For emphasis, we summarize all assumptions in \Cref{thm:main-intro} below. 

\begin{assumption}\label{assm:funcClass}
Let $f:\mathbb{R}^n\rightarrow\mathbb{R}$ be a function such that it  
\begin{enumerate}[label={(\alph*)}]
    \item is bounded below, namely, $\inf_{\mathbb{R}^n}f > -\infty$;
    \item is locally Lipschitz continuous on $\mathbb{R}^n$; see \Cref{def:locally-lipschitz}; 
    \item admits a chain rule; see \Cref{def:chain-rule}; 
    \item has finitely many critical values; see \Cref{subsec:clarke}; 
    \item has bounded subgradient trajectories; see \Cref{def:bdd-subgrad}. 
\end{enumerate}
\end{assumption}

\begin{proof}[Proof of \Cref{thm:main-intro}]
Let $f:\mathbb{R}^n\rightarrow \mathbb{R}$ be a function satisfying \Cref{assm:funcClass} If $f$ has no spurious setwise local minima, then $f$ has no spurious local minima. We next prove the converse. Let $S\subset\mathbb{R}^n$ be a setwise local minimum of $f$. We seek to show that $S$ is a setwise global minimum of $f$. If $S^\circ = \emptyset$, then by \Cref{lem:bdd-const} $f(x) = \sup_S f$ for all $x \in \partial S = \bar{S} \setminus S^\circ = S$ since $S$ is closed by \Cref{def:setwise-local-min}. Thus $f$ is constant on $S$. By definition of setwise local minima (\Cref{def:setwise-local-min}), there exists an open subset $U$ of $\mathbb{R}^n$ containing $S$ such that the constant value of $f$ on $S$ is less than or equal to $f(y)$ for all $y \in U \setminus S$. Thus every point in $S$ is a local minimum of $f$. Since every local minimum of $f$ is a global minimum, $\inf_S f = \inf_{\mathbb{R}^n} f$ and $S$ is a setwise global minimum according to \Cref{def:setwise-global-min}. The rest of the proof deals with the case when $S^\circ \neq \emptyset$. Let $C$ be the set of all critical points of $f$ in $S$ and consider the following optimization problem:
\begin{equation}\label{eq:optimization}
    \inf_{x \in C} f(x).
\end{equation}
We claim that the set of (global) solutions to \Cref{eq:optimization} is nonempty, and that any solution is a local minimum of $f$ belonging to the setwise local minimum $S$. We first show that the feasible set of \Cref{eq:optimization} is nonempty.

Since $S^\circ \neq \emptyset$, let $x_0 \in S^\circ$. If $x_0 \in C$, then the feasible set $C$ is nonempty. We thus assume that $x_0 \notin C$. Since $f$ is locally Lipschitz and bounded below, by \Cref{prop:exist} there exists a subgradient trajectory $x:[0,\infty)\rightarrow\mathbb{R}^n$ starting at $x_0$. We next show that $x([0,\infty)) \subset S$. We reason by contradiction and assume that $S^c \cap x([0,\infty)) \neq \emptyset$, where $S^c$ is the complement of $S$ in $\mathbb{R}^n$. Then $S^\circ$ and $S^c$ are disjoint open subsets of $\mathbb{R}^n$ such that $S^\circ \cap x([0,\infty)) \neq \emptyset$ (the intersection contains $x_0$), $S^c \cap x([0,\infty)) \neq \emptyset$, and $x([0,\infty)) = (x([0,\infty)) \cap S^\circ) \cup (x([0,\infty)) \cap S^c) \subset \mathbb{R}^n \setminus \partial S$ (since\footnote{If there exists $t>0$ such that $f(x(t)) = f(x(0))$, then $f(x(t)) - f(x(0)) = \int_0^{t} \|x'(s)\|^2ds = 0$ and $x'(s) = 0$ for almost every $s\in (0,t)$. Since $x'(s) \in -\partial f(x(s))$ for almost every $s>0$, by \cite[2.1.5 Proposition (b) p. 29]{clarke1990} we have $0 \in \partial f(x(0))$.} $f(x(t)) < f(x(0)) = f(x_0) \leqslant f(x)$ for all $t > 0$ and $x \in \partial S$, where the last inequality follows from \Cref{lemma:boundary}). Thus the connected set $x([0,\infty))$ is the union of two relatively open disjoint nonempty sets, which is a contradiction. 

Since $f$ has bounded subgradient trajectories and $x(\cdot)$ is an arbitrary subgradient trajectory starting at $x_0$, by \Cref{def:bdd-subgrad} and without loss of generality there exists $r>0$ such that $\|x(t)\| \leqslant r$ for all $t\geqslant 0$. We next show that there exists a critical point of $f$ in $B(0,r) \cap S$. Suppose that there exist two constants $T,\epsilon>0$ for which $\lVert x'(t)\rVert\geqslant\epsilon$ for all $t\geqslant T$ such that $x'(t)\in-\partial f(x(t))$. By \cite[Lemma 5.2]{davis2020stochastic}, we have $(f \circ x)'(t) = -\lVert x'(t)\rVert^2 \leqslant -\epsilon^2$ for almost every $t\geqslant T$. By integrating, we get $f(x(t)) - f(x(T)) \leqslant - \epsilon^2t$ and thus $f(x(t))$ converges to $-\infty$ as $t\to\infty$. This is impossible since $x(t) \in B(0,r)$ and $f$ is continuous. Hence there exists a time sequence $t_k\to\infty$ such that $\lVert x'(t_k)\rVert\to 0$ as $k\to\infty$ and $x'(t_k)\in-\partial f(x(t_k))$ for all $k\in\mathbb{N} := \{0,1,2,\hdots\}$. By the Bolzano–Weierstrass theorem, there exists a subsequence $x(t_{k_j})$ of $x(t_k)$ such that $x(t_{k_j})\to\tilde{x}\in\mathbb{R}^n$ as $j\to\infty$. Since $x'(t_{k_j})\in -\partial f(x(t_{k_j}))$, by \cite[2.1.5 Proposition (b) p. 29]{clarke1990} we have $0\in -\partial f(\tilde{x})$. Finally, since $x([0,\infty)) \subset S$ and $S$ is closed, we have $\tilde{x}\in C$. We obtain that $C\neq \emptyset$ as desired.

Since $f$ has finitely many critical values and $C \neq \emptyset$, the set of solutions to \Cref{eq:optimization} is nonempty. Let $x^* \in C$ be a solution, that is to say $f(x^*) = \min_C f$. Recall that $C$ is a subset of the setwise local minimum $S$. If $x^*$ is a local minimum of $f$, then it is a global minimum of $f$ in $S$ since every local minimum of $f$ is a global minimum. Thus $\inf_S f = \inf_{\mathbb{R}^n} f$ and $S$ is a setwise global minimum. For the remainder of the proof, we consider the case where $x^*$ is not a local minimum and show that this leads to a contradiction. We first show that there exists $s_0\in S^\circ$ such that $f(s_0)<f(x^*)$. This is clearly true if $x^*\in S^\circ$ since one can then find a ball centered at $x^*$ inside $S^\circ$. If $x^*\in S \setminus S^\circ = \partial S$, then we reason by contradiction and assume that $f(x) \geqslant f(x^*)$ for all $x\in S^\circ$. By \Cref{lemma:boundary}, we have $f(x) = f(x^*) = \sup_S f \geqslant f(y) \geqslant f(x^*)$ for all $(x,y)\in (S \setminus S^\circ) \times S^\circ$. Hence $f(x^*) = f(x)$ for all $x\in S$. Since $S$ is a setwise local minimum, there exists an open set $U$ such that $f(x) \geqslant f(x^*)$ holds for all $x\in U\setminus S$. Thus $f(x)\geqslant f(x^*)$ for all $x\in U$ and $x^*$ is a local minimum. This yields a contradiction. Hence let $s_0\in S^\circ$ be such that $f(s_0)<f(x^*)$. The nonempty closed set $S' := S \cap [f \leqslant (f(s_0)+f(x^*))/2]$ is a setwise local minimum of $f$ where $[f\leqslant \alpha] := \{ x \in \mathbb{R}^n ~\vert~ f(x) \leqslant \alpha \}$. Indeed, for all $x \in S'$ and $y \in U \setminus S' = (U \setminus S) \cup (U \setminus [f \leqslant (f(s_0)+f(x^*))/2])$, we have\footnote{Indeed, for any sets $A,B$, and $C$ it holds that $A\setminus (B\cap C) = (A \setminus B) \cup (A\setminus C)$.} $f(x) \leqslant (f(s_0)+f(x^*))/2 \leqslant f(y)$. Since $s_0 \in S^\circ$ and $f(s_0)<(f(s_0)+f(x^*))/2$, we have $s_0 \in S^\circ \cap [f\leqslant (f(s_0)+f(x^*))/2]^\circ = (S \cap [f\leqslant (f(s_0)+f(x^*))/2]))^\circ = (S')^\circ$. Hence the setwise local minimum $S'$ has nonempty interior. Also, $S' \subset S$ and $\sup_{S'} f < f(x^*) = \min_C f$ where we remind the reader that $C$ is the set of critical points in $S$. Thus $S'$ is devoid of critical points. However, by the previous
paragraph, setwise local minima of $f$ with nonempty interior must contain a critical point. This yields a contradiction.
\end{proof}

\begin{remark}[Finitely many critical values]
This assumption is not intuitive and we explain why it is necessary by the following example. Define $f:\mathbb{R}\rightarrow\mathbb{R}$ as   
\begin{equation*}
    f(x) := \begin{cases}
        (x+4)^2-8 & \text{if } x\leqslant -2; \\
        -x^2 & \text{if } x\in[-2,0]; \\
        -2^{-k}(x-2k)^{2^{k+1}}-3(1-2^{-k}) & \text{if } x\in [2k,2k+1],\; k\in\mathbb{N}; \\
        2^{-k}(x-2k)^{2^{k+1}}-3(1-2^{-k}) & \text{if } x\in [2k-1,2k],\; k\in\mathbb{N}^+. 
    \end{cases}
\end{equation*}
where $\mathbb{N}^+$ is the set of all positive integers. To be more intuitive, we give the plot of $f$ on $[-7,7]$ in \Cref{fig:infinite-cvalues}. 
\begin{figure}[ht]
    \centering
    \includegraphics[width=.7\textwidth]{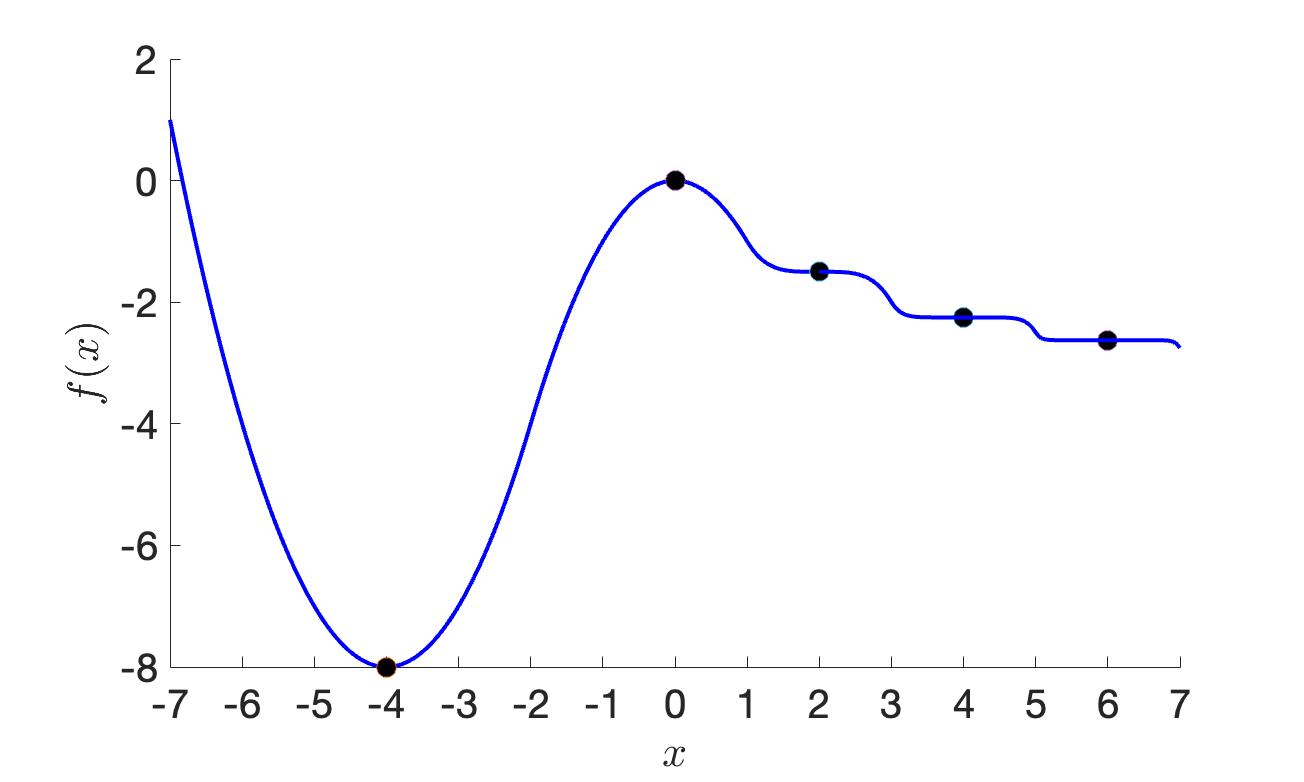}
    \caption{An example of function with infinitely many critical values}
    \label{fig:infinite-cvalues}
\end{figure}

By standard calculus, one can see $f$ is continuously differentiable, $f(x)\to -3$ as $x\to\infty$, and $f(x)\geqslant -8$ over $\mathbb{R}$. Furthermore, $\{-4\}\cup\{2k\}_{k\in\mathbb{N}}$ are all critical points of $f$, with critical values $\{-8\}\cup \{-3(1-2^{-k})\}_{k\in\mathbb{N}}$ respectively. Finally, the subgradient trajectory of $f$ starting at $x_0<0$ will converge to the critical point $x=-4$; the one starting at $x_0=0$ will stay at the critical point $x=0$; and the one starting at $x_0>0$ such that $2k<x_0\leqslant 2k+2$ will converge to $x=2k+2$, for all $k\in\mathbb{N}$. This shows $f$ has bounded subgradient trajectories. Thus, $f$ satisfies all conditions in \Cref{assm:funcClass} except the finiteness of critical values. It is also easy to see $f$ has no spurious local minima because all of its critical points are either global minimum ($x=-4$), or local maximum ($x=0$), or saddle points. However, for any $a>0$, the set $[a,\infty)$ is a spurious local minima at infinity. This shows that \Cref{thm:main-intro} may not hold for functions with infinitely many critical values. 
\end{remark}

\begin{remark}[Bounded subgradient trajectories]
This is the main assumption of \Cref{thm:main-intro}. Without it, one could easily think of a smooth function without any spurious local minimum, yet has spurious local minimum at infinity. This is the case of the function in \Cref{fig:setwise} in which the yellow curve corresponds to an unbounded gradient trajectory. In order to prove the necessity of the boundedness assumption, it suffices to consider the univariate function $f$ defined in \cite[Figure 4(a)]{joszneurips2018} defined by
\begin{equation*}
    f(x) := \frac{x^2(1+x^2)}{1+x^4},\quad f'(x) = -\frac{2x(x^4-2x^2-1)}{(x^4+1)^2}. 
\end{equation*}
By solving $f'(x)=0$, we know that $f$ has three critical points, among which $x=0$ is the global minimum and $x=\pm (\sqrt{2}-1)^{-1/2}$ are two global maxima. Thus, $f$ is bounded below, continuously differentiable (hence locally Lipschitz and admits a chain rule), has finitely many critical values, and has no spurious local minima. Since $f$ is strictly decreasing for all $x\geqslant (\sqrt{2}-1)^{-1/2}\approx 1.55$ and $f(x)\to 1$ as $x\to\infty$, one can easily see $[2,\infty)$ is a spurious local minimum at infinity. This shows that \Cref{thm:main-intro} does not hold and the reason is that $f$ does not have bounded subgradient trajectories. To see this explicitly, consider the Cauchy problem 
\begin{equation*}
    \dot{x} = \frac{2x(x^4-2x^2-1)}{(x^4+1)^2}, \quad x(0)=2. 
\end{equation*}
By using separation of variables, the unique solution $x(t)$ is given by 
\begin{align*}
    c+2t = \frac{1}{4}x^4+x^2 &+(2+\sqrt{2})\log(x^2-\sqrt{2}-1) \nonumber\\
    & +(2-\sqrt{2})\log(x^2+\sqrt{2}-1)-\log x =: g(x), 
\end{align*}
where $c$ is a constant determined by $x(0)=2$. It is easy to see that $x$ is strictly increasing so $x(t)\geqslant 2$ for all $t\in[0,\infty)$. Note that $g$ is continuous on $[2,\infty)$, so if $x$ is bounded, then $g\circ x$ is bounded. This contradicts the fact that $g(x(t))=2t+c\to\infty$ as $t\to\infty$, and thus $f$ has an unbounded subgradient trajectory. 
\end{remark}

\section{Applications}\label{sec:applications}

In this section, we use \Cref{thm:main-intro} to analyze the landscape of some widely used loss functions in unconstrained optimization. To be more specific, we will consider deep linear neural network, one dimensional deep sigmoid neural network, matrix sensing, and nonsmooth matrix factorization in the following four subsections respectively. 

\subsection{Deep linear neural network}

As a prototypical example in deep learning, the landscape of deep linear neural network has been widely studied; see for example \cite{kawaguchi2016deepNIPS,laurent2018deep,venturi2019spurious}. Consider minimizing the loss function of linear neural network without bias term 
\begin{equation}\label{eq:linearNN-obj}
    f(W_1,\ldots,W_L) := \frac{1}{2}\lVert W_L\cdots W_1 X-Y\rVert_F^2, 
\end{equation}
where $X\in\mathbb{R}^{d_0\times d_x}$, $Y\in\mathbb{R}^{d_L\times d_x}$, and $W_i\in\mathbb{R}^{d_i\times d_{i-1}}$ for $i=1,\ldots,L$. Here $\|\cdot\|_F$ denotes the Frobenius norm. 
It was recently established that $f$ has no spurious valleys \cite[Theorem 11]{venturi2019spurious}, however this fact alone does not imply the absence of spurious local minima at infinity (recall \Cref{fig:relu}). Together with the fact that $f$ has no spurious local minima \cite[Corollary 1]{zhang2019depth} and that $f$ is semi-algebraic, it can be deduced that $f$ has no spurious setwise local minima (and thus no spurious local minima at infinity). 

The proof of the absence of spurious valleys \cite[Theorem 11]{venturi2019spurious} is tailored to the problem at hand. Using linear algebra, it argues that from any initial point one can construct a piecewise linear path to a global minimum along which the objective function is non-increasing. The proof spans multiple pages and requires several technical lemmas. The proof that we propose is shorter and follows a general principle, namely Theorem \ref{thm:main-intro}, that applies to various problems as the next subsections will show. The first four assumptions of Theorem \ref{thm:main-intro} are easy to verify: $f$ is nonnegative, hence bounded below; $f$ is continuously differentiable, hence locally Lipschitz and admits a chain rule; $f$ is semi-algebraic, by \cite[Corollary 1.1]{pham2016genericity}, it has finitely many critical values. Thus, it suffices to show $f$ has bounded gradient trajectories.

\begin{proposition}\label{prop:linearNN-bddgrad}
Linear neural network with loss function \Cref{eq:linearNN-obj} has bounded gradient trajectories. 
\end{proposition}

An existing proof of \Cref{prop:linearNN-bddgrad} under additional assumptions on network structure, initialization, input data, or target data can be found, for instance, in \cite{bah2022learning,eftekhari2020training,chen2022non}. To the best of our knowledge, the closest result to \Cref{prop:linearNN-bddgrad} is \cite[Theorem 3.2]{bah2022learning}, which shows that gradient trajectories are bounded if $XX^{\rm T}$ is of full rank. In the proof of \Cref{prop:linearNN-bddgrad}, we show that this rank assumption on $X$ can be removed and hence \Cref{prop:linearNN-bddgrad} applies to any linear neural network.

\begin{proof}[Proof of \Cref{prop:linearNN-bddgrad}] Since $f$ is locally Lipschitz and lower bounded, by \\ \Cref{prop:exist} there exists a gradient trajectory for any initial point.
By \cite[Lemma 2.1]{bah2022learning}, the gradient trajectories of $f$ satisfy the initial value problem 
\begin{subequations}\label{linearNN_dyn1}
    \begin{align}
        &\dot{W}_i = -(W_L\cdots W_{i+1})^{\rm T}(W_L\cdots W_1 X-Y)(W_{i-1}\cdots W_1X)^{\rm T}, \label{linearNN_dyn1_contr} \\
        & W_i(0)=W_i^0,\quad W_i^0\in\mathbb{R}^{d_i\times d_{i-1}} \text{ is a given constant matrix},
    \end{align}
\end{subequations}
for all $i=1,\ldots,L$. Note that if $i=L$, \Cref{linearNN_dyn1_contr} reduces to 
\begin{equation*}
    \dot{W}_L=-(W_L\cdots W_1 X-Y)(W_{L-1}\cdots W_1X)^{\rm T},
\end{equation*}
and if $i=1$, \Cref{linearNN_dyn1_contr} reduces to 
\begin{equation*}
    \dot{W}_1=-(W_L\cdots W_2)^{\rm T}(W_L\cdots W_1 X-Y)X^{\rm T}.
\end{equation*}
\vskip 1ex
Note that \cite[Theorem 3.2]{bah2022learning} proved the boundedness of gradient trajectories of $f$ when $XX^{\rm T}$ is invertible. Thus, we only need to show we can always reduce the boundedness of gradient trajectories of $f$ for general $X$ to the boundedness of gradient trajectories of another function $g$ in the same form as $f$ but with invertible $XX^{\rm T}$. Let $X=U\Sigma V^{\rm T}$ be a singular value decomposition, where $U\in\mathbb{R}^{d_0\times d_0}$ and $V\in\mathbb{R}^{d_x\times d_x}$ are orthogonal matrices, and $\Sigma\in\mathbb{R}^{d_0\times d_x}$ is a rectangular matrix satisfying 
\begin{equation*}
    \Sigma = \begin{bmatrix}
        \Lambda & 0 \\
        0 & 0
    \end{bmatrix},\quad \Lambda=\diag(\lambda_1,\ldots,\lambda_r)\succ 0,
\end{equation*}
where $r\leqslant \min\{d_0,d_x\}$. Eliminating $X$ in \Cref{linearNN_dyn1_contr}, it reduces to 
\begin{align*}
    \dot{W}_i &= -(W_L\cdots W_{i+1})^{\rm T}(W_L\cdots W_1 U\Sigma V^{\rm T}-Y)(W_{i-1}\cdots W_1U\Sigma V^{\rm T})^{\rm T} \\
    &= -(W_L\cdots W_{i+1})^{\rm T}(W_L\cdots W_1 U\Sigma-YV)(W_{i-1}\cdots W_1U\Sigma)^{\rm T}.
\end{align*}
Define $Z:=YV\in\mathbb{R}^{d_L\times d_x}$, and \Cref{linearNN_dyn1} reduces to 
\begin{subequations}\label{linearNN_dyn2}
    \begin{align}
        &\dot{W}_i = -(W_L\cdots W_{i+1})^{\rm T}(W_L\cdots W_1 U\Sigma-Z)(W_{i-1}\cdots W_1U\Sigma)^{\rm T},\label{linearNN_dyn2_contr} \\
        & W_i(0)=W_i^0,\quad \forall\,i=1,\ldots L.
    \end{align}
\end{subequations}
Denote $\overline{W}_1:=W_1U\in\mathbb{R}^{d_1\times d_0}$ and $\overline{W}_1^0:=W_1^0U\in\mathbb{R}^{d_1\times d_0}$. To keep the notation consistent, also let $\overline{W}_i:=W_i$ and $\overline{W}_i^0:=W_i^0$ for $i=2,\ldots,L$. Thus, \Cref{linearNN_dyn2} reduces to 
\begin{subequations}\label{linearNN_dyn3}
    \begin{align}
        &\dot{\overline{W}}_i = -(\overline{W}_L\cdots \overline{W}_{i+1})^{\rm T}(\overline{W}_L\cdots \overline{W}_1 \Sigma-Z)(\overline{W}_{i-1}\cdots \overline{W}_1\Sigma)^{\rm T}, \\
        & \overline{W}_i(0)=\overline{W}_i^0,\quad \forall\,i=1,\ldots L.
    \end{align}
\end{subequations}
Partition the matrices $\overline{W}_1$, $\overline{W}_1^0$, and $Z$ into two column blocks: 
\begin{equation*}
    \overline{W}_1 = \begin{bmatrix}
        \overline{W}_{11} & \overline{W}_{12}
    \end{bmatrix},\quad \overline{W}_1^0 = \begin{bmatrix}
        \overline{W}_{11}^0 & \overline{W}_{12}^0
    \end{bmatrix},\quad Z = \begin{bmatrix}
        Z_1 & Z_2 
    \end{bmatrix},
\end{equation*}
where $\overline{W}_{11}$, $\overline{W}_{11}^0$, and $Z_1$ consist of the first $r$ columns of $\overline{W}_1$, $\overline{W}_1^0$ and $Z$ respectively. Thus, when $i=1$, \Cref{linearNN_dyn3} can be reduced into
\begin{align*}
    &\dot{\overline{W}}_{11} = -(\overline{W}_L\cdots \overline{W}_2)^{\rm T}(\overline{W}_L\cdots \overline{W}_2\overline{W}_{11}\Lambda-Z_1)\Lambda^{\rm T}, \quad \dot{\overline{W}}_{12} = 0, \\ 
    &\overline{W}_{11}(0) = \overline{W}_{11}^0,\quad \overline{W}_{12}(0) = \overline{W}_{12}^0.
\end{align*}
When $i=2,\ldots,L$, \Cref{linearNN_dyn3} can be reduced into 
\begin{align*}
    &\dot{\overline{W}}_i = -(\overline{W}_L\cdots \overline{W}_{i+1})^{\rm T}(\overline{W}_L\cdots \overline{W}_2\overline{W}_{11} \Lambda-Z_1)(\overline{W}_{i-1}\cdots \overline{W}_2\overline{W}_{11}\Lambda)^{\rm T}, \\
    & \overline{W}_i(0)=\overline{W}_i^0.
\end{align*}
It indicates that $\overline{W}_{12}(t)=\overline{W}_{12}^0$ for all $t\geqslant 0$. Denote $\widetilde{W}_1:=\overline{W}_{11}$ and  $\widetilde{W}_1^0:=\overline{W}_{11}^0$. To keep the notation consistent, also let $\widetilde{W}_i:=\overline{W}_i$ and $\widetilde{W}_i^0:=\overline{W}_i^0$ for $i=2,\ldots,L$. Therefore, \Cref{linearNN_dyn3} reduces to 
\begin{subequations}\label{linearNN_dyn4}
    \begin{align}
        &\dot{\widetilde{W}}_i = -(\widetilde{W}_L\cdots \widetilde{W}_{i+1})^{\rm T}(\widetilde{W}_L\cdots \widetilde{W}_1\Lambda-Z_1)(\widetilde{W}_{i-1}\cdots \widetilde{W}_1\Lambda)^{\rm T}, \\
        & \widetilde{W}_i(0)=\widetilde{W}_i^0,\quad \forall\,i=1,\ldots L.
    \end{align}
\end{subequations}
Define the new function $g$ as 
\begin{equation*}
    g(\widetilde{W}_1,\ldots,\widetilde{W}_L) := \frac{1}{2}\lVert \widetilde{W}_L\cdots \widetilde{W}_1\Lambda-Z_1\rVert_F^2. 
\end{equation*}
Notice that the gradient trajectories of $g$ satisfy \Cref{linearNN_dyn4}. To prove $f$ has bounded gradient trajectories, it is equivalent to prove $g$ has bounded gradient trajectories, because $\lVert W_1\rVert_F=\lVert W_1U\rVert_F=\lVert \overline{W}_1\rVert_F$ and $\lVert \overline{W}_1(t)\rVert_F^2=\lVert \widetilde{W}_1(t)\rVert_F^2+\lVert \overline{W}_{12}^0\rVert_F^2$ for all $t\geqslant 0$. Since $\Lambda\Lambda^{\rm T}$ is invertible, by \cite[Theorem 3.2]{bah2022learning}, $g$ has bounded gradient trajectories, and so does $f$. 
\end{proof}

With \Cref{prop:linearNN-bddgrad}, we verified that $f$ satisfies \Cref{assm:funcClass}. Thus, \Cref{eq:linearNN-obj} has no spurious setwise local minima if and only if it has no spurious local minima. Since a local minimum at infinity is an unbounded setwise local minimum, and $f$ has no spurious local minima, we conclude that $f$ has no spurious local minima at infinity. This proves the first result in \Cref{cor:main}.

\subsection{One dimensional deep sigmoid neural network}

Though famous for its benign theoretical properties, linear neural network is rarely used in practice because of its low representation power. We want to take a step further in the case of nonlinear deep neural network. In this subsection, we focus on neural network with sigmoid activation function in one dimensional case. Landscape analysis of one or two-hidden layer sigmoid neural network can be found, for instance, in \cite{venturi2019spurious,ding2022suboptimal,li2022benefit}. However, none of the results above can be easily generalized to arbitrary many layers. 

Consider minimizing the following loss function of sigmoid neural network
\begin{equation}\label{eq:sigmoidNN-obj}
    f(w_1,\ldots,w_L) := \frac{1}{2}(w_L\sigma(w_{L-1}\cdots \sigma(w_1x))-y)^2, 
\end{equation}
where $\sigma(z):= (1+e^{-z})^{-1}$ is the sigmoid function and $w_i,x,y\in\mathbb{R}$ for all $i=1,\ldots,L$. We want to apply  \Cref{thm:main-intro} to conclude that \Cref{eq:sigmoidNN-obj} has no spurious setwise local minimum, and hence no local minima at infinity. Again, the first three assumptions in \Cref{assm:funcClass} are easy to verify: $f$ is nonnegative, hence bounded below; $f$ is continuously differentiable, hence locally Lipschitz, and admits a chain rule. Note that $f$ is not semi-algebraic, but it is definable in the real exponential field \cite{wilkie1996model} \cite[Section 6.2]{bolte2020conservative}, so by Morse–Sard theorem for definable functions \cite[Corollary 9(ii)]{bolte2007clarke}, it has finitely many critical values. 

Again, it remains to show \Cref{eq:sigmoidNN-obj} has bounded gradient trajectories. However, the techniques in the proof of  \Cref{prop:linearNN-bddgrad} cannot be adapted to this case because the auto-balancing property in \cite[Theorem 2.1]{du2018algorithmic} does not hold. Surprisingly, it is still true that \Cref{eq:sigmoidNN-obj} has bounded gradient trajectories. 

\begin{proposition}\label{prop:sigmoidNN-bddgrad}
One dimensional sigmoid neural network with loss function \Cref{eq:sigmoidNN-obj} has bounded gradient trajectories. 
\end{proposition}
\begin{proof}
Since $f$ is locally Lipschitz and lower bounded, by \Cref{prop:exist} there exists a gradient trajectory for any initial point. For simplicity, define $p_i$ for $i=0,\ldots,L-2$ recursively by $p_0:=x$, $p_1:=\sigma(w_1x)$ and $p_{i+1}:=\sigma(w_{i+1}p_i)$. The gradient trajectories of $f$ satisfy 
\begin{subequations}
    \begin{align}
    \dot{w}_L &= -(w_L\sigma(w_{L-1} p_{L-2})-y)\sigma(w_{L-1} p_{L-2}),  \\
    \dot{w}_i &= \frac{p_{i-1}}{1+e^{w_ip_{i-1}}}\dot{w}_{i+1}w_{i+1},\quad i=L-1,\ldots,1. \label{eq:sigmoid_middle}
\end{align}
\end{subequations}
We will prove each $w_i$ is bounded inductively from the last layer to the first layer. The relation between the last two layers $w_L$ and $w_{L-1}$, and the relation between the first two layers can be regarded as the base cases. 
\vskip 1ex
We claim that there exists a time $T$ such that $\dot{w}_i$ and $w_i$ does not change sign for all $t\geqslant T$ and for all $i$. To verify this, first notice that the claim is true for the last layer, i.e., $\dot{w}_L$ and $w_L$ will not change sign for all $t\geqslant T$.  Suppose $\dot{w}_L$ changes sign, by continuity and mean value theorem, there exists $t^*>0$ such that $\dot{w}_L(t^*)=0$. However, $\dot{w}_L(t^*)=0$ implies $\dot{w}_i(t^*)=0$ for all $i$, meaning that a critical point is achieved and the gradient trajectory is stopped for all $t\geqslant t^*$. In this case, all $w_i$'s are trivially bounded. Thus, we assume the trajectory will never stop at a finite time. In this case, either $\dot{w}_L(t)>0$ or $\dot{w}_L(t)<0$ for all $t\geqslant 0$. Since $w_L$ is monotonic, it either keeps the sign unchanged or changes the sign only once. Thus, there exists $T_L>0$ such that $w_L$ does not change sign on $[T_L,\infty)$. Notice that for all $i\geqslant 2$, $p_{i-1}(t)\in(0,1)$ for all $t\geqslant 0$. Since $\dot{w}_Lw_L$ does not change sign on $[T_L,\infty)$, \Cref{eq:sigmoid_middle} implies that $\dot{w}_{L-1}$ does not change sign on $[T_L,\infty)$ either. Therefore, we conclude that $w_{L-1}$ is monotonic. Similarly, there exists $T_{L-1}>T_L$ such that $\dot{w}_{L-1}$ and $w_{L-1}$ does not change sign on $[T_{L-1},\infty)$. Recursively using the above argument, we can show the claim is true for all $i\geqslant 2$ on $[T_2,\infty)$. For $i=1$, although $p_0=x$ may not be in $(0,1)$, since $x$ is a constant, the fact that $\dot{w}_2$ and $w_2$ do not change sign still implies that $\dot{w}_1$ does not change sign and hence there exists $T_1>T_2$ such that $w_1$ does not change sign on $[T_1,\infty)$. Therefore, the claim holds for $i=1,\ldots,L$ by choosing $T=T_1$. 
\vskip 1ex
By the claim proved in the last paragraph, for $i=1,\ldots,L$, either $\dot{w}_iw_i$ is nonnegative or $\dot{w}_iw_i$ is negative on $[T,\infty)$. Now we are going to prove each $w_i$ is bounded. The first step is to prove the last two layers $w_L$ and $w_{L-1}$ are bounded. Consider the case where $\dot{w}_Lw_L$ is nonnegative on $[T,\infty)$. \Cref{eq:sigmoid_middle} implies that $\dot{w}_{L-1}\geqslant 0$ and $w_{L-1}$ is increasing over $[T,\infty)$, so there exists a constant $c_{L-1}$ such that $w_{L-1}(t)\geqslant c_{L-1}$ for all $t\geqslant 0$. Since $p_{L-2}\in(0,1)$, we have $\sigma(w_{L-1}p_{L-2})\geqslant \sigma(-\vert c_{L-1}\vert)>0$. Again, by \cite[Lemma 5.2]{davis2020stochastic}, $\frac{d}{dt}f(w_1,\ldots,w_L)\leqslant 0$ and $f(w_1,\ldots,w_L)\leqslant C$ for some constant $C$ on $[0,\infty)$. Thus, it is easy to see $\vert w_L\vert\sigma(w_{L-1}p_{L-2})\leqslant C_1$ for some constant $C_1$ on $[0,\infty)$. Since $\sigma(w_{L-1}p_{L-2})\in [\sigma(-\vert c_{L-1}\vert),1)$, we conclude $\vert w_L\vert$ is bounded. Suppose $w_{L-1}$ is unbounded. Since it is increasing and does not change sign, $w_{L-1}(t)>0$ for all $t\geqslant T$ and $w_{L-1}(t)\to \infty$ as $t\to\infty$. By \Cref{eq:sigmoid_middle}, 
\begin{equation}
    \dot{w}_{L-1} = \frac{p_{L-2}}{1+e^{w_{L-1}p_{L-2}}}\dot{w}_Lw_L \leqslant \dot{w}_Lw_L \label{eq:sigmoid_middle_up}, 
\end{equation}
because $p_{L-2}\in(0,1)$ and $1+e^{w_Lp_{L-2}}>1$. By \Cref{eq:sigmoid_middle_up}, $w_{L-1}-\frac{1}{2}w_L^2$ is a decreasing function on $[T,\infty)$. Hence, $w_{L-1}-\frac{1}{2}w_L^2\leqslant C_2$ for some constant $C_2$. Notice that $w_L$ is bounded but $w_{L-1}(t)\to\infty$ as $t\to\infty$, so a contradiction occurs. Therefore, $w_{L-1}$ is bounded. 
\vskip 1ex
Now we consider the case where $\dot{w}_Lw_L$ is negative on $[T,\infty)$. In this case, \Cref{eq:sigmoid_middle} implies $\dot{w}_{L-1}\leqslant 0$, so $w_{L-1}$ is decreasing on $[T,\infty)$ and there exists a constant $d_{L-1}$ such that $w_{L-1}\leqslant d_{L-1}$. Since $p_{L-2}/(1+e^{w_{L-1}p_{L-2}})\in(0,1)$ and $\dot{w}_Lw_L\leqslant 0$ on $[T,\infty)$, we have $\dot{w}_{L-1}\geqslant \dot{w}_Lw_L$. This shows $w_{L-1}-\frac{1}{2}w_L^2$ is increasing on $[T,\infty)$, and hence $w_{L-1}\geqslant \tilde{d}_{L-1}$ for some constant $\tilde{d}_{L-1}$. Therefore, $w_{L-1}\in[\tilde{d}_{L-1},d_{L-1}]$ is bounded. By exactly the same argument as in the case when $\dot{w}_Lw_L$ is nonnegative, we know $\sigma(w_{L-1}p_{L-2})\in [\sigma(-\vert\tilde{d}_{L-1}\vert),1)$ and $w_L$ is bounded by using the boundedness of objective function $f$. 
\vskip 1ex
Up to now, we have proved boundedness for the last two layers $w_L$ and $w_{L-1}$. For $i=2,\ldots,L-2$, by discussing two cases $\dot{w}_{i+1}w_{i+1}\geqslant 0$ and $\dot{w}_{i+1}w_{i+1}\leqslant 0$, together with the boundedness of $w_{i+1}$, we can prove that $w_i$ is bounded by exactly the same argument as we did in the last two paragraphs. The induction starts with proving $w_{L-2}$ is bounded and ends with proving $w_2$ is bounded. Once we prove $w_2$ is bounded, consider the relation between $w_1$ and $w_2$, 
\begin{equation*}
    (1+e^{w_1x})\dot{w}_1 = x\dot{w}_2w_2. 
\end{equation*}
If $x=0$, then $\dot{w}_1=0$ implies $w_1$ is a constant over $[0,\infty)$, so it must be bounded. Suppose $x\ne 0$, by taking integration with respect to $t$ and multiplying $x$ on both sides, we have 
\begin{equation*}
    w_1x+e^{w_1x} = \frac{x^2}{2}w_2^2 + C_3. 
\end{equation*}
Let $z=w_1x$, then $z+e^z\to \pm\infty$ as $z\to\pm\infty$. Thus, the boundedness of $w_2$ implies the boundedness of $z=w_1x$. Since $x\ne 0$ is a constant, $w_1$ is bounded. Therefore, we proved that $w_i$ is bounded for all $i=1,\ldots,L$. 
\end{proof}

With \Cref{prop:sigmoidNN-bddgrad}, we can conclude that $f$ has no spurious setwise local minimum if and only if it has no spurious local minima. However, from the gradient of $f$, we can easily see that any critical point of it will be a global minimum, so $f$ has neither  spurious local minimum nor spurious setwise local minimum. This verifies the second result in \Cref{cor:main}. 

Unfortunately, unlike linear neural networks, the result in \Cref{prop:sigmoidNN-bddgrad} is not true in general even in one-hidden layer case, if more than one data point is given; see \Cref{eg:2data-sigmoid}. However, it is still an open question whether the gradient trajectories will be bounded in the over-parameterized case (in which case there exists at least one achievable global minimum). 

\begin{example}\label{eg:2data-sigmoid}
Consider the following function 
\begin{equation}\label{eq:sigmoid-counter}
    f(w_1,w_2):=\frac{1}{2}[(w_2\sigma(w_1)-1)^2+(w_2\sigma(-w_1)+1)^2]. 
\end{equation}
The above function represents a one-hidden layer sigmoid neural network with two data $(x_1,y_1)=(1,1)$ and $(x_2,y_2)=(-1,-1)$. By directly computing the gradient, one can easily see that \Cref{eq:sigmoid-counter} has only one critical point $(0,0)$ which is a strict saddle with $f(0,0)=1$. The global minimum is asymptotically attained as $w_1\to\pm\infty$ and $w_2\to 1-2(1+e^{2w_1})^{-1}$, and its corresponding objective value approaches to $1/2$. In this case, the gradient trajectory of \Cref{eq:sigmoid-counter} starting at any point $x_0$ such that $f(x_0)<1$ must be unbounded. 
\end{example}

\subsection{Matrix sensing}

Matrix sensing is a widely used model in computer vision and statistics; see for instance \cite{chi2019nonconvex,recht2010guaranteed}. Given $r\geqslant 1$, the goal is to recover an unknown target matrix $M\in\mathbb{R}^{n_1\times n_2}$ of rank less than or equal to $r$ from a set of linear measurements $b_i=\langle A_i,M\rangle_F$, where $A_i\in\mathbb{R}^{n_1\times n_2}$ for $i=1,\ldots,m$ are sensing matrices and $\langle \cdot,\cdot\rangle_F$ is the Frobenius inner product. In order to do so, we minimize the mean square loss 
\begin{equation}\label{eq:recovery-obj}
    f(X,Y) := \frac{1}{2m}\sum_{i=1}^m(\langle A_i,XY^{\rm T}\rangle_F-b_i)^2. 
\end{equation}
where $X\in\mathbb{R}^{n_1\times r}$ and $Y\in\mathbb{R}^{n_2\times r}$.
The landscape of \Cref{eq:recovery-obj} has been studied widely, for example, in \cite{zhu2018global,park2017non,li2020global}. Most of these work are based on the restrictive isometry property (RIP) of sensing matrices. A set of sensing matrices $A_i$ for $i=1,\ldots,m$ are said to have $(r,\delta_r)$-RIP \cite{recht2010guaranteed} if there exists $\delta_r\in(0,1)$ such that 
\begin{equation*}
    (1-\delta_r)\|\widetilde{M}\|_F^2 \leqslant \frac{1}{m}\sum_{i=1}^m\langle A_i,\widetilde{M}\rangle_F^2 \leqslant (1+\delta_r)\|\widetilde{M}\|_F^2
\end{equation*}
holds for any matrix $\widetilde{M}$ with $\rank(\widetilde{M})\leqslant r$. To the best of our knowledge, the minimal assumptions to guarantee no spurious local minima for \Cref{eq:recovery-obj} is for the sensing matrices to satisfy $(4r,\delta_{4r})$-RIP with $\delta_{4r}\leqslant 1/5$, as proposed in \cite[Theorem III.1]{li2020global}. 

However, \Cref{thm:main-intro} is applicable to matrix sensing under a weaker condition than RIP. The first four assumptions in \Cref{assm:funcClass} hold because of exactly the same reasons as in the linear neural network case. Thus, it suffices to show \Cref{eq:recovery-obj} has bounded gradient trajectories. A sufficient condition is to require the sensing matrices to be \textit{lower bounded}, i.e., there exists a constant $c>0$ such that for any matrix $\widetilde{M}\in\mathbb{R}^{n_1\times n_2}$ with $\rank(\widetilde{M})\leqslant r$, 
\begin{equation*}
    \frac{1}{m}\sum_{i=1}^m\langle A_i,\widetilde{M}\rangle_F^2 \geqslant c\lVert \widetilde{M}\rVert_F^2. 
\end{equation*}
It is easy to see any level of RIP will imply the existence of such a constant $c$. 

\begin{proposition}\label{prop:ms-bddgrad}
Matrix sensing with loss function \Cref{eq:recovery-obj} and lower bounded sensing matrices has bounded gradient trajectories. 
\end{proposition}
\begin{proof}
Since $f$ is locally Lipschitz and lower bounded, by \Cref{prop:exist} there exists a gradient trajectory for any initial point. The gradient trajectories of $f$ satisfy the initial value problem
\begin{align*}
    &\dot{X} = -\frac{1}{m}\sum_{i=1}^m(\langle A_i,XY^{\rm T}\rangle_F-b_i)A_iY, \\
    &\dot{Y} = -\frac{1}{m}\sum_{i=1}^m(\langle A_i,XY^{\rm T}\rangle_F-b_i)A_i^{\rm T}X, \\
    &X(0) = X_0,\quad Y(0) = Y_0. 
\end{align*}
Notice that $\dot{X}^{\rm T}X=Y^{\rm T}\dot{Y}$ and $X^{\rm T}\dot{X}=\dot{Y}^{\rm T}Y$, so 
\begin{equation*}
    \frac{d}{dt}(X^{\rm T}X-Y^{\rm T}Y) = \dot{X}^{\rm T}X+X^{\rm T}\dot{X} - \dot{Y}^{\rm T}Y - Y^{\rm T}\dot{Y} = 0. 
\end{equation*}
This implies that $X^{\rm T}X-Y^{\rm T}Y=C$ where $C\in\mathbb{R}^{r\times r}$ is a constant. Since the function value is decreasing along gradient trajectories \cite[Lemma 5.2]{davis2020stochastic}, there exists a constant $c_1$ such that $f(X(t),Y(t))\leqslant c_1$ for all $t\geqslant 0$. Combined with the assumption that sensing matrices are lower bounded, there exist constants $c$ and $c_2$ such that
\begin{align*}
    c\lVert XY^{\rm T}\rVert_F^2 &\leqslant\frac{1}{m}\sum_{i=1}^m\langle A_i,XY^{\rm T}\rangle_F^2\leqslant \frac{1}{m}\sum_{i=1}^m[2(\langle A_i,XY^{\rm T}\rangle_F-b_i)^2+2b_i^2] \\
    &= 2f(X,Y)+\frac{2}{m}\sum_{i=1}^mb_i^2 \leqslant 2c_1+\frac{2}{m}\sum_{i=1}^mb_i^2 =: c_2. 
\end{align*}
We have $\lVert XY^{\rm T}\rVert_F^2\leqslant c_3 := c_2/c$. Notice that 
\begin{equation*}
    \lVert X^{\rm T}X\rVert_F^2+\lVert Y^{\rm T}Y\rVert_F^2 = \lVert X^{\rm T}X-Y^{\rm T}Y\rVert_F^2+2\lVert XY^{\rm T}\rVert_F^2 \leqslant \lVert C\rVert_F^2+2c_3. 
\end{equation*}
Define the constant $c_4:=2c_3+\lVert C\rVert_F^2$. By the Cauchy-Schwarz inequality, 
\begin{equation*}
    \lVert X\rVert_F^4+\lVert Y\rVert_F^4\leqslant \rank(X)\lVert X^{\rm T}X\rVert_F^2+\rank(Y)\lVert Y^{\rm T}Y\rVert_F^2 \leqslant (n_1+n_2+r)c_4. 
\end{equation*}
Thus, $X$ and $Y$ are bounded. 
\end{proof}

Therefore, \Cref{thm:main-intro} says that matrix sensing has no spurious setwise local minima if and only if it has no spurious local minima, given that the sensing matrices are lower bounded. Equipped with $(4r,\delta_{4r})$-RIP where $\delta_{4r}\leqslant 1/5$, we conclude that matrix sensing has no spurious local minima at infinity, as shown in the third statement in \Cref{cor:main}.

\subsection{Nonsmooth matrix factorization}

In this subsection, we consider the application of \Cref{thm:main-intro} in a nonsmooth setting, namely, the nonsmooth matrix factorization problem. We consider minimizing the loss function 
\begin{equation}\label{eq:nonsmooth-obj}
    f(X,Y) := \| XY^{\rm T} - M \|_1, 
\end{equation}
where $X\in\mathbb{R}^{m\times r}$, $Y\in\mathbb{R}^{n\times r}$ are decision variables and $M\in\mathbb{R}^{m\times n}$ is the given data matrix. Here $\|A\|_1:=\sum_{i=1}^m\sum_{j=1}^n |A_{ij}|$ for any $A\in\mathbb{R}^{m\times n}$. In robust principal component analysis (PCA) problem with sparse noise, \Cref{eq:nonsmooth-obj} is usually used as a surrogate function for the original $\ell_0$-norm formulation; see \cite{gillis2018complexity,charisopoulos2021low}. There are few landscape results of \Cref{eq:nonsmooth-obj} in the general rank case.  However, if $\rank(M)=1=r$, \Cref{eq:nonsmooth-obj} is shown to have no spurious local minima if every entry $M_{ij}$ of $M$ is nonzero \cite{josz2021nonsmooth}. 

It is hard to analyze \Cref{eq:nonsmooth-obj} because it is nonsmooth, nonconvex, and noncoercive. Despite all those ``non'' properties, we show that \Cref{thm:main-intro} is still applicable to \Cref{eq:nonsmooth-obj} without any rank assumption on $M$. As a corollary, when $\rank(M)=1$ and every entry of $M$ is nonzero, \Cref{eq:nonsmooth-obj} has no spurious setwise local minimum, hence no spurious local minima at infinity. Again, the first four assumptions in \Cref{assm:funcClass} are easy to check: $f$ is bounded below because it is nonnegative; $f$ is locally Lipschitz because of \cite[Theorem 2.3.10]{clarke1990}; since $f$ is semi-algebraic, by \cite[Corollary 5.4]{drusvyatskiy2015curves} and \cite[Corollary 1.1]{pham2016genericity}, it admits a chain rule and has finitely many critical values. 

To verify \Cref{eq:nonsmooth-obj} has bounded subgradient trajectories, we discover that the auto-balancing property in \cite[Theorem 2.2]{du2018algorithmic} also holds for nonsmooth matrix factorization. The result can be summarized in the following proposition. 

\begin{proposition}\label{prop:NMF-bddsgrad}
Nonsmooth matrix factorization with loss function \Cref{eq:nonsmooth-obj} has bounded subgradient trajectories. 
\end{proposition}
\begin{proof}
Since $f$ is locally Lipschitz and lower bounded, by \Cref{prop:exist} there exists a subgradient trajectory for any initial point. Let $(X_0,Y_0)\in\mathbb{R}^{m\times r}\times\mathbb{R}^{n\times r}$. Consider an absolutely continuous function $Z:[0,\infty)\rightarrow\mathbb{R}^{m\times r}\times \mathbb{R}^{n \times r}$ such that 
\begin{equation*}
    Z'(t) \in -\partial f(Z(t)),~~~\text{for almost every}~ t \geqslant 0, ~~~ \text{and}~ Z(0) = (X_0,Y_0).
\end{equation*}
By \cite[Theorem 2.3.10]{clarke1990}, 
\begin{equation*}
    \partial f(X,Y) = \left\{ \left. \begin{pmatrix}
        \Lambda Y \\ \Lambda^T X 
    \end{pmatrix} ~\right\vert~ \Lambda \in \mathrm{sign}(XY^T-M) \right\}
\end{equation*}
where $\mathrm{sign}$ is an element-wise operation mapping each entry of a matrix to a real number in $[-1,1]$ such that 
\begin{equation*}
\text{sign}(x) :=
\left\{
\begin{array}{cl}
-1 & \text{if} ~ x < 0, \\
\big[-1,1\big] & \text{if} ~ x = 0, \\
1 & \text{if} ~ x > 0.
\end{array}
\right.
\end{equation*}
Hence, with $Z =: (X,Y)$, for almost every $t\geqslant 0$ we have
\begin{subequations}
    \begin{gather}\label{eq:derivativeInclusionXYLambda}
    X'(t) = - \Lambda(t) Y(t), ~~~ Y'(t) = - \Lambda(t)^T X(t), \\
    \Lambda(t) \in \mathrm{sign}(X(t)Y(t)^T-M). 
\end{gather}
\end{subequations}
Consider $\phi:[0,\infty)\rightarrow\mathbb{R}$ defined by $\phi(t) := X(t)^TX(t)-Y(t)^TY(t)$. By taking derivative, we have 
\begin{equation}\label{eq:1stderivativePhi}
    \phi'(t) = X'(t)^TX(t) + X(t)^TX'(t)-Y'(t)^TY(t)-Y(t)^TY'(t). 
\end{equation}
Combining \Cref{eq:derivativeInclusionXYLambda} and \Cref{eq:1stderivativePhi}, we have 
\begin{align*}
    \phi'(t) &= -Y(t)^T \Lambda(t)^T X(t) - X(t)^T\Lambda(t) Y(t) \\
    &\qquad +X(t)^T\Lambda(t)Y(t)+Y(t)^T\Lambda(t)^TX(t) = 0. 
\end{align*}
Hence the continuous function $\phi$ is constant on $[0,\infty)$. Also, we have
\begin{align*}
    \|X^TX-Y^TY\|_F^2 & = \|X^TX\|_F^2 + \|Y^TY\|_F^2 - 2 \langle X^TX,Y^TY \rangle_F \\
    & = \|X^TX\|_F^2 + \|Y^TY\|_F^2 - 2\|XY^T\|_F^2 \\
    & \geqslant \|X^TX\|_2^2 + \|Y^TY\|_2^2 - 2\|XY^T\|_F^2 \\
    & = \|X\|_2^4 + \|Y\|_2^4 - 2 \|XY^T\|_F^2 \\
    & \geqslant \|X\|_2^4 + \|Y\|_2^4 - 2mn \|XY^T\|_1^2 \\
    & \geqslant \|X\|_2^4 + \|Y\|_2^4 - 2mn(\|XY^T-M\|_1+\|M\|_1)^2. 
\end{align*}
Here $\|\cdot\|_2$ denotes the spectral norm. Therefore, for all $t\geqslant 0$, we have
\begin{align*}
    \|X(t)\|_2^4 + \|Y(t)\|_2^4 \leqslant & \|X(t)^TX(t)-Y(t)^TY(t)\|_F^2 \\
    &\qquad +2mn(\|X(t)Y(t)^T-M\|_1+\|M\|_1)^2 \\
    \leqslant & \|X_0^TX_0-Y_0^TY_0\|_F^2 \\ 
    &\qquad +2mn(\|X_0Y_0^T-M\|_1+\|M\|_1)^2.\qedhere
\end{align*}
\end{proof}

Combined with \Cref{prop:NMF-bddsgrad}, \Cref{thm:main-intro} shows that \Cref{eq:nonsmooth-obj} has no spurious setwise local minimum if and only if it has no spurious local minima. Under the condition in \cite[Theorem 1]{josz2021nonsmooth}, i.e., $\rank(M)=1=r$ and all the entries of $M$ are non-zero, \Cref{eq:nonsmooth-obj} reduces to 
\begin{equation}\label{eq:NMFrk1}
    f(x,y) := \sum_{i=1}^m\sum_{j=1}^n \vert x_iy_j-M_{ij}\rvert, 
\end{equation}
where $x\in\mathbb{R}^m$ and $y\in\mathbb{R}^n$. In this case, \Cref{eq:NMFrk1} has no spurious local minima, thus it has no spurious local minima at infinity, and we obtain the last result in \Cref{cor:main}.

\section*{Acknowledgments}
We thank the reviewers and the associate editor for their valuable feedback.

\appendix

\section{Proof of \texorpdfstring{\Cref{lem:bdd-const}}{}}\label{app:lem:bdd-const}
Let $S \subset \mathbb{R}^n$ be a setwise local minimum of a continuous function $f:\mathbb{R}^n\rightarrow\mathbb{R}$. Let $U\supset S$ be an open set such that $f(x)\leqslant f(y)$ for all $x \in S$ and $y\in U\setminus S$. Note that $S$ is closed, so its boundary is defined by $\partial S := S\setminus S^\circ$. Let $z\in\partial S$ and consider any real number $\epsilon>0$. Since $f(z)+(-\epsilon,\epsilon)$ is a neighborhood of $f(z)$, by continuity of $f$, there exists a neighborhood $N(z)$ of $z$ such that $f(N(z))\subset f(z)+(-\epsilon,\epsilon)$. Since $U$ is a neighborhood of $z$, $N'(z):=U \cap N(z)$ is also a neighborhood of $z$ with $f(N'(z))\subset f(z)+(-\epsilon,\epsilon)$. The set $N'(z)\cap S$ is nonempty because $z \in S$ and the set $N'(z)\setminus S$ is nonempty because $z\in\partial S$. For any $x \in N'(z) \cap S$ and $y \in N'(z)\setminus S$, it follows that
\begin{equation*}
     \inf_{U \setminus S} f -\epsilon \leqslant f(y)-\epsilon < f(z) < f(x)+\epsilon \leqslant \sup_S f+\epsilon \leqslant \inf_{U\setminus S} f +\epsilon. 
\end{equation*}
The last inequality follows from the definition of setwise local minima. As $\epsilon>0$ was arbitrary, we deduce that 
\begin{equation*}
\label{inf_sup}
    \inf_{U \setminus S} f = f(z) = \sup_S f.
\end{equation*}
Thus, $f$ is a constant on the boundary of $S$ and $f$ attains its maximum over $S$ on the boundary of $S$.

\section{Proof of \texorpdfstring{\Cref{prop:valley}}{}}\label{app:prop:valley}
\begin{enumerate}[label=(\alph*), listparindent=20pt]
    \item Let $S$ be a setwise local minimum. By \cref{lem:bdd-const}, we know that $c:=\sup_S f=f(z)$ for all $z\in\partial S$. Take a path-connected component $C$ of $S$. Then $C\subset [f\leqslant c] := \{ x \in \mathbb{R}^n ~\vert~ f(x) \leqslant c \}$. Since $C$ is path-connected, there exists a path-connected component $V$ of $[f\leqslant c]$ such that $C\subset V$. By definition, $V$ is a valley. This shows that a path-connected component of a setwise local minimum is a subset of a valley. 
    
    If in addition, $S$ is a strict setwise local minimum, then we distinguish two cases. If $S = \mathbb{R}^n$, then the path-connected component $C$ of $S$ is equal to $\mathbb{R}^n$ and is therefore a valley. Otherwise, it suffices to show that $V\subset S$. Indeed, $V$ is then a path-connected subset of $S$ containing the path-connected component $C$ of $S$, so that by maximality, $V=C$. Therefore $C$ is valley.
    
    Consider an open set $U\supset S$ such that $f(x)>f(y)$ for all $x\in U\setminus S$ and $y\in S$. In order to show that $V\subset S$, it suffices to show that $V\cap(U\setminus S)=\emptyset$ and $V\cap U^c=\emptyset$ because if so, then $V=(V\cap S)\cup(V\cap(U\setminus S))\cup(V\cap U^c)=V\cap S$. Since $f(w)\leqslant c$ for all $w\in V$ and $f(w)>c$ for all $w\in U\setminus S$ (the supremum function value $f(y)=c$ can be attained by some $y\in \partial S\subset S$), we know that $V\cap (U\setminus S)=\emptyset$. Thus, $V=(V\cap S)\cup(V\cap U^c)$. Note that $V\cap S$ is nonempty and closed because $C\subset V\cap S$ and $V$ and $S$ are both closed. Since $V\cap U^c$ is also closed and $V$ is connected, $V\cap U^c$ must be empty. 
    
    \item Let $f:\mathbb{R}^n\rightarrow\mathbb{R}$ be a continuous function and $a\in \mathbb{R}$ be a nonempty sublevel set of $f$. By continuity of $f$, $[f\leqslant a]$ is closed in $\mathbb{R}^n$. Suppose $[f\leqslant a]$ has finitely many connected components $C_1,\ldots,C_k$. Denote $\overline{B}$ as the closure of any set $B\subset\mathbb{R}^n$. Since $C_i$'s are connected, by \cite[Theorem 23.4]{munkrestopology}, $\overline{C_i}$'s are also connected. Since $C_i\subset [f\leqslant a]$, $\overline{C_i}\subset \overline{[f\leqslant a]} = [f\leqslant a]$. By \cite[Theorem 25.1]{munkrestopology}, $\overline{C_i}$ has no intersection with any other $C_j$ for $j\ne i$. Together with the fact that $[f\leqslant a]=\bigcup_{i=1}^kC_i$, we have $\overline{C_i}\subset C_i$, and hence $\overline{C_i}=C_i$. Thus, each $C_i$ is closed in $\mathbb{R}^n$. 
    
    For any fixed $i$, denote $C_{-i}:=[f\leqslant a]\setminus C_i$, then $C_{-i}=\bigcup_{j=1,j\ne i}^k C_j$ is a closed set disjoint with $C_i$. By \cite[Theorem 32.2]{munkrestopology}, there exist disjoint open sets $D,E\subset\mathbb{R}^n$ such that $C_i\subset D$ and $C_{-i}\subset E$. Take $U=D$ in \Cref{def:setwise-local-min}, then $f(x)\leqslant a$ for all $x\in C_i$ because $C_i\subset [f\leqslant a]$. Furthermore, $f(y)>a$ for all $y\in U\setminus C_i$ because $(U\setminus C_i)\cap [f\leqslant a]=\emptyset$. This verifies that $C_i$ is a strict setwise local minimum of $f$. 
\end{enumerate}

\section{Proof of \texorpdfstring{\Cref{prop:char-coercive}}{}}\label{app:prop:char-coercive}
Let $S$ be a spurious setwise local minimum at infinity. Since $\inf_S f > \inf_{\mathbb{R}^n} f$, it must be that $S\neq \mathbb{R}^n$. By \Cref{def:setwise-local-min}, there exists $y \in S^c$ such that $S\subset \{x\in\mathbb{R}^n\,\vert\,f(x)\leqslant f(y)\}$. Since $f$ is coercive, its sublevel sets are bounded and hence $S$ is bounded. $S$ is thus not a spurious local minimum at infinity.

\section{Proof of \texorpdfstring{\Cref{prop:exist}}{}}\label{app:prop:exist}
For a fixed real number $\tau>0$, define a sequence $x_k^\tau$ recurrently by letting $x_0^\tau:=x_0$ and 
\begin{equation*}
    x_{k+1}^\tau \in \arg\min_{x\in\mathbb{R}^n} \left\{f(x)+\frac{\lVert x-x_k^\tau\rVert^2}{2\tau}\right\}, \quad \forall k\in\mathbb{N}. 
\end{equation*}
A solution exists because $f$ is bounded below and the objective function is coercive. Any solution satisfies 
\begin{equation*}
    v_{k+1}^\tau := \frac{x_{k+1}^\tau-x_k^\tau}{\tau} \in -\partial f(x_{k+1}^\tau), \quad \forall k\in\mathbb{N}. 
\end{equation*}
Define two functions $x^\tau,\tilde{x}^\tau:\mathbb{R}_+\rightarrow \mathbb{R}^n$ where $\mathbb{R}_+:=[0,\infty)$ by
\begin{equation*}
    x^\tau(t) := x_{k+1}^\tau,\quad \tilde{x}^\tau(t) := x_k^\tau + (t-k\tau)v_{k+1}^\tau,\quad \forall t\in(k\tau,(k+1)\tau]
\end{equation*}
for all $k\in\mathbb{N}$, with initial condition $x^\tau(0)=\tilde{x}^\tau(0)=x_0$. Note that $\tilde{x}^\tau$ is absolutely continuous because it is piecewise affine. On the contrary, $x^\tau$ is not continuous. Also, define $v^{\tau}:\mathbb{R}_+\rightarrow\mathbb{R}^n$ by 
\begin{equation*}
    v^\tau(t) := v_{k+1}^\tau, \quad \forall t\in(k\tau,(k+1)\tau],\quad \forall k\in\mathbb{N}, 
\end{equation*}
and choose $v^\tau(0)\in -\partial f(x_0)$. Since  $(\tilde{x}^\tau)'=v^\tau$ on $(k\tau,(k+1)\tau)$ for all $k \in \mathbb{N}$, and $v^\tau(t)\in -\partial f(x^\tau(t))$ for all $t\geqslant 0$, we conclude that $(\tilde{x}^\tau)'(t)\in -\partial f(x^\tau(t))$ for almost every $t\in\mathbb{R}_+$. By optimality of $x_{k+1}^\tau$, we have 
\begin{equation*}
    f(x_{k+1}^\tau) + \frac{\lVert x_{k+1}^\tau-x_k^\tau\rVert^2}{2\tau} \leqslant f(x_k^\tau), \quad \forall k\in\mathbb{N}. 
\end{equation*}
For any $l\in\mathbb{N}$, we have 
\begin{equation*}
    \sum_{k=0}^l\frac{\lVert x_{k+1}^\tau-x_k^\tau\rVert^2}{2\tau} \leqslant f(x_0^\tau) - f(x_{l+1}^\tau) \leqslant f(x_0) - \inf_{\mathbb{R}^n} f =: C < \infty
\end{equation*}
since $f$ is bounded below. Observe that 
\begin{equation*}
    \sum_{k=0}^l\frac{\lVert x_{k+1}^\tau-x_k^\tau\rVert^2}{2\tau} = \sum_{k=0}^l\frac{\tau}{2}\lVert v_{k+1}^\tau\rVert^2 = \frac{1}{2}\sum_{k=0}^l\int_{k\tau}^{(k+1)\tau}\left\lVert (\tilde{x}^\tau)'(t)\right\rVert^2dt. 
\end{equation*}
Fix $T \geqslant 0$ from now on. From the above, we have
\begin{equation}\label{eq:L2-bounded}
    \int_0^T\left\lVert (\tilde{x}^\tau)'(t)\right\rVert^2dt \leqslant \sum_{k=0}^{\lfloor T/\tau\rfloor}\int_{k\tau}^{(k+1)\tau}\left\lVert (\tilde{x}^\tau)'(t)\right\rVert^2dt \leqslant 2C. 
\end{equation}
Since $\tilde{x}^\tau$ is absolutely continuous, for any $s,t\in [0,T]$ we have
\begin{subequations}\label{eq:1/2-holder}
    \begin{align}
    \|\tilde{x}^\tau(t)-\tilde{x}^\tau(s)\| &= \left\|\int_s^t(\tilde{x}^\tau)'(u)\;du\right\| \\
    &\leqslant \left(\int_0^T\left\lVert (\tilde{x}^\tau)'(t)\right\rVert^2dt\right)^{1/2}\vert t-s\vert^{1/2} \leqslant \sqrt{2C}\vert t-s\vert^{1/2} \label{eq:1/2-holder_cs} 
\end{align}
\end{subequations}
where we use the Cauchy-Schwarz inequality. Now one can see $(\tilde{x}^\tau)_{\tau>0}$ is a family of uniformly bounded and equicontinuous functions on the compact interval $[0,T]$. Therefore, by Arzel\`a-Ascoli theorem \cite[Theorem 7.25]{rudin1964principles}, there exists a sequence of positive reals $(\tau_k)_{k\in\mathbb{N}}$ such that $\tau_k\to 0$ and $\tilde{x}^{\tau_k}\to x^*$ uniformly on $[0,T]$ as $k\to\infty$. For all $k\in \mathbb{N}$ and $t\in (k\tau,(k+1)\tau]$, we have $\tilde{x}^\tau((k+1)\tau)=x_k^\tau + \tau v^\tau_{k+1} = x^\tau_{k+1} = x^\tau(t)$. Thus $\|\tilde{x}^\tau(t)-x^\tau(t)\| = \|\tilde{x}^\tau(t)-\tilde{x}^\tau((k+1)\tau)\| \leqslant \sqrt{2C}\tau^{1/2}$ for all $t\in [0,T]$ where the inequality is due to \Cref{eq:1/2-holder} (take $s:=(k+1)\tau$).
Combined with the fact that $\tilde{x}^{\tau_k}\to x^*$ uniformly on $[0,T]$, one can see that $x^{\tau_k}\to x^*$ uniformly on $[0,T]$. Since \cref{eq:L2-bounded} implies that $((\tilde{x}^{\tau_k})')_{k\in\mathbb{N}}$ is a bounded sequence in $L^2([0,T],\mathbb{R}^n)$, there exists a subsequence $(\tau_{k_j})_{j\in\mathbb{N}}$ such that $(\tilde{x}^{\tau_{k_j}})'\to v^*$ weakly in $L^1([0,T],\mathbb{R}^n)$ as $j\to\infty$ by \cite[Corollary 14 p. 413]{Fitzpatrick2010}. Since $\tilde{x}^{\tau_{k_j}}$ is absolutely continuous, for all $t\in[0,T]$, we have 
\begin{equation*}
    \tilde{x}^{\tau_{k_j}}(t) - \tilde{x}^{\tau_{k_j}}(0) = \int_0^t(\tilde{x}^{\tau_{k_j}})'(u)\;du. 
\end{equation*}
Take $j\to\infty$ on both sides, we have 
\begin{equation*}
    x^*(t) - x^*(0) = \int_0^tv^*(u)\;du, 
\end{equation*}
where the convergence of the integral relies on the fact that the constant functions equal to the canonical basis of $\mathbb{R}^n$ lie in $L^\infty([0,T],\mathbb{R}^n)$. Thus, $x^*$ is absolutely continuous and $(x^*)'(t)=v^*(t)$ for almost every $t\in[0,T]$. Recall that for all $k\in \mathbb{N}$, it holds for almost every $t\in[0,T]$ that
\begin{equation*}
    (\tilde{x}^{\tau_k})'(t) = v^{\tau_k}(t) \in -\partial f(x^{\tau_k}(t)). 
\end{equation*}
Since $f$ is locally Lipschitz, the set-valued function $-\partial f$ is upper semicontinuous \cite[2.1.5 Proposition (d) p. 29]{clarke1990} with nonempty compact values \cite[2.1.2 Proposition (a) p. 27]{clarke1990}, hence proper upper hemicontinuous \cite[Proposition 1 p. 60]{aubin1984differential}. In addition, $x^{\tau_k}\to x^*$ uniformly on $[0,T]$ and $(\tilde{x}^{\tau_k})'\to (x^*)'$ weakly in $L^1([0,T],\mathbb{R}^n)$. Therefore,  $(x^*)'(t)\in-\partial f(x^*(t))$ for almost all $t\in[0,T]$ by \cite[Theorem 1 p. 60]{aubin1984differential}\footnote{In the theorem we take $F:=-\partial f$, $X=Y:=\mathbb{R}^n$, and $I:=[0,T]$.}. The initial condition also holds since $\tilde{x}^{\tau}(0)=x_0$ for all $\tau>0$. 
\vskip 1ex
We have proved that for any initial point $x_0$, there exists $x^*:[0,T]\rightarrow\mathbb{R}^n$ such that $(x^*)'(t)=-\partial f(x^*(t))$ holds for almost every $t\in [0,T]$ with any $T>0$. Since $T$ is independent of $x_0$, by setting $T=1$, there exists a sequence of absolutely continuous functions $(x_k)_{k\in\mathbb{N}}$ such that 
\begin{equation*}
    x_k'(t)\in -\partial f(x_k(t)), \quad \mathrm{for~a.e.}~t \in [0,1], \quad x_k(0) = x_{k-1}(1), 
\end{equation*}
for all $k\in\mathbb{N}$ where $x_{-1}(0)=x_0$. Therefore, the desired function $x:[0,\infty)\rightarrow \mathbb{R}^n$ can be defined in a piecewise fashion by 
\begin{equation*}
    x(t) := x_k(t-k),\quad t\in[k,k+1),\quad \forall k\in\mathbb{N}. 
\end{equation*}
By construction, $x$ is absolutely continuous on any compact interval $[a,b]\subset[0,\infty)$. 

\bibliographystyle{abbrv}    
\bibliography{references}

\end{document}